\documentclass[11pt]{amsart}
\usepackage{amssymb,latexsym, amsmath, amsxtra}
\usepackage{graphicx}

\theoremstyle{plain}
\newtheorem{theorem}{Theorem}[section]

\newtheorem{lemma}[theorem]{Lemma}

\theoremstyle{definition}

\theoremstyle{remark}

\numberwithin{equation}{section}

\DeclareMathOperator*{\Res}{Res}

\begin{document}

\title{Symplectic $n$-level densities with restricted support}

\abstract
In this paper we demonstrate that the alternative form, derived by us in an earlier paper, of the $n$-level densities for eigenvalues of matrices from the classical compact group $USp(2N)$ is far better suited for comparison with derivations of the $n$-level densities of zeros in the family of Dirichlet $L$-functions associated with real quadratic characters than the traditional determinantal random matrix formula.  Previous authors have found ingenious proofs that the leading order term of the $n$-level density of the zeros agrees with the determinantal random matrix result under certain conditions, but here we show that comparison is more straightforward if the more suitable form of the random matrix result is used.  For the support of the test function in $[1,-1]$ and in $[-2,2]$ we compare with existing number theoretical results.  For support in $[-3,3]$ no rigorous number theoretical result is known for the $n$-level densities, but we derive the densities here using random matrix theory in the hope that this may make the path to a rigorous number theoretical result clearer.  
\endabstract

\author {A.M. Mason}
\address{NDM Experimental Medicine, University of Oxford, John Radcliffe Hospital, Oxford, OX3 9DU, United Kingdom} \email{amy.mason@ndm.ox.ac.uk}

\author{N.C. Snaith}
\address{School of Mathematics,
University of Bristol, Bristol, BS8 1TW, United Kingdom}
\email{N.C.Snaith@bris.ac.uk}

\thanks{
2010 Mathematics Subject Classification: 11M50, 15B52, 11M26, 11G05, 11M06, 15B10\\
The first author was supported by an EPSRC Doctoral Training Grant.  The second author was supported at the ICERM Semester Program on ``Computational Aspects of the Langlands Program" during the majority of this work. } \maketitle


\section{Introduction}

In the 1990s Katz and Sarnak \cite{kn:katzsarnak99a,kn:katzsarnak99b} considered statistics of zeros of $L$-functions near the critical point (the point at which the real axis crosses the critical line on which the complex zeros are expected to lie).  They predict that in a natural family of $L$-functions these zeros would, when averaged across the family, display the same statistical behaviour as the eigenvalues near 1 of matrices chosen at random with respect to Haar measure from the classical compact groups $U(N)$, $O(N)$ or $USp(2N)$. To see this correspondence with random matrix theory it is necessary to scale the zeros by their mean density and a natural asymptotic limit is taken.  Under these conditions, local statistics of the zeros are expected to match the equivalent statistics of eigenvalues of random matrices of large dimension.  

There has been a large body of work aiming to prove the Katz-Sarnak correspondence under various conditions and for various families of $L$-functions. Some of these papers consider just the leading order terms of the statistics in the asymptotic limit (for example \cite{kn:alpmil14,kn:duemil06,kn:entrodrud,kn:fouiwa03,kn:gao14,kn:guloglu05,kn:ILS99,kn:levmil13,kn:mil04,kn:milpec12,kn:ozlsny99,kn:rub01,kn:young}).  Others (for example \cite{kn:fiomil15,kn:Goesetal,kn:huymilmor,kn:mil07,kn:mil09,kn:milmon}) match lower order terms with conjectures for the lower than leading order behaviour given by, for instance, the Ratios Conjectures  \cite{kn:cfz2,kn:consna06}. Ratios conjectures combine knowledge  from random matrix theory with non-rigourous number theoretical arguments to produce very precise formulae for averages of ratios of $L$-functions, and hence the zero statistics that can be derived from these ratios.  

Here we will concentrate just on the leading order behaviour of the $n$-level densities of the zeros of one particular family of $L$-functions, although the random matrix formulae presented can be compared equally well with any family of $L$-functions with statistics like $USp(2N)$.  The family is that of Dirichlet $L$-functions associated with real quadratic characters.  Members of the family are given, for fundamental discriminants $d$, by the Dirichlet series
\begin{equation}
L(s,\chi_d)=\sum_{n=1}^\infty \frac{\chi_d(n)}{n^s},
\end{equation}
valid for $\Re s >1$, where $\chi_d(n)$ is Kronecker's symbol.  In keeping with the literature on this subject, we denote the set of $d$'s in our family by $D(X)=\{d \;{\rm fundamental \; discriminant}: X/2\leq |d|<X\}$, although from the random matrix point of view the full range from 0 to $X$ would also be acceptable.  

The quest to provide evidence for the Katz-Sarnak philosophy for this family dates back to the work of 
 {\"O}zl{\"u}k and Snyder \cite{kn:ozlsny99} who demonstrated that the 1-level density for the family of quadratic Dirichlet $L$-functions matches, in the scaling limit, with the 1-level density of eigenvalues from large $USp(2N)$ matrices.  The method for obtaining level density statistics of zeros involves sampling a test function at the positions of the zeros.  All the known rigorous results require a restriction on the support of the Fourier transform of this test function.  In the case of \cite{kn:ozlsny99}, the Fourier transform of the test function has support in $[-2,2]$.  
 
 This work was extended to all $n$-level densities by Rubinstein \cite{kn:rub01}.  The test function in this case takes $n$ variables and Rubinstein writes the Fourier transform as $\prod_{i=1}^n \hat{f_i}(u_i)$ and requires it to be supported on $\sum_{i=1}^n |u_i|<1$, where
\begin{equation}\label{eq:fourier}
\hat{f}(u)=\int_{-\infty}^{\infty} f(x) e^{2\pi i x u} dx.
\end{equation}

Gao \cite{kn:gao14} succeeded in writing down the $n$-level densities when the support is extended to $\sum_{i=1}^n |u_i|<2$, but was only able to verify that these agree, at leading order, with random matrix theory for $n\leq 3$.  Levinson and Miller \cite{kn:levmil13} extended this to $n\leq 7$.

Recently Entin, Roditty-Gershon and Rudnick \cite{kn:entrodrud} proved that all $n$-level densities match with random matrix theory at leading order for $\sum_{i=1}^n |u_i|<2$ by an ingenious method using comparison with the equivalent statistic for function field zeta functions.  

One reason that there has been so much difficulty in confirming that the scaling limit of the $n$-level density agrees with random matrix theory is that until recently the most obvious form of the $n$-level density of $USp(2N)$ eigenvalues that number theorists had to compare their results to was the determinantal formula (written for simplicity for a test function $f$ that is symmetric under permutation of all variables): 

\begin{eqnarray}\label{eq:determinantal}
&&\lim_{N\rightarrow \infty} \int_{USp(2N)} \sum_{\substack{j_1,\ldots,j_n\\ {\rm distinct}}} f\Big(\tfrac{N}{\pi}\theta_{j_1}, \ldots, \tfrac{N}{\pi}\theta_{j_n} \Big)d\mu_{{\rm Haar}}\nonumber \\
&&=\lim_{N\rightarrow \infty}\frac{1}{n!} \int_{0}^\pi \cdots \int_{0}^\pi f\Big(\tfrac{N}{\pi}\theta_{j_1}, \ldots, \tfrac{N}{\pi}\theta_{j_n} \Big) \det_{n\times n} \Big(K_{USp(2N)}(\theta_k,\theta_j)\Big) d\theta_1 \cdots d\theta_n\nonumber \\
&&=\frac{1}{n!} \int_{0}^\infty \cdots \int_0^\infty f(\theta_1, \ldots, \theta_n) \det_{n\times n} K_{USp} (\theta_k,\theta_j) d\theta_1\cdots d\theta_n,
\end{eqnarray}
where
\begin{equation}
K_{USp(2N)} (x,y) = S_{2N+1}(y-x)-S_{2N+1}(y+x)
\end{equation}
with
\begin{equation}
S_N(x)=\frac{1}{2\pi} \frac{\sin Nx/2}{\sin x/2},
\end{equation}
and
\begin{equation}
K_{USp}(x,y) = S(y-x)+S(y+x)
\end{equation}
with
\begin{equation}
S(x)=\frac{\sin \pi x}{\pi x}.
\end{equation}
In the above the eigenvalues of a matrix $A\in USp(2N)$ are denoted $e^{\pm i\theta_1},\ldots,e^{\pm i\theta_N}$.

The derivation of formula (\ref{eq:determinantal}) is a beautiful piece of mathematics (see for example \cite{kn:mehta} or \cite{kn:conrey04}), and the determinantal structure is very important in random matrix theory, but this formulation causes two problems when trying to compare it with $n$-level densities of zeros of $L$-functions. 

The first difficulty is the condition that the sum is over distinct eigenvalues.  In number theory, the sums naturally occur as sums over unrestricted $n$-tuples of zeros.  Therefore the first step in comparing a number theoretical result with random matrix theory is to do combinatorial sieving to remove duplicate zeros from the sum.  This can be achieved, but adds an extra layer of complication. 

The second difficulty is that the number theoretical $n$-level densities don't naturally take the structure of a determinant.  It is the task of trying to artificially create the determinantal form, or alternatively of unpicking the determinantal form on the random matrix side, that has made this problem so challenging.  

Conrey and Snaith \cite{kn:consna08} demonstrated for the case of unitary matrices, that if one forgoes the beautiful machinery that leads to the determinantal form of the $n$-level density and instead proceeds via averages of ratios of characteristic polynomials of $U(N)$ matrices then you arrive at a different formula for the densities (also called correlation functions).  They show \cite{kn:consna14} that with the alternative formula it is much more straightforward to reproduce the work in which Rudnick and Sarnak \cite{kn:rudsar} show that for test functions with limited support, the leading order correlations of the zeros of automorphic $L$-functions are the same as those of eigenvalues of random unitary matrices.  Proceeding via conjectures for averages of ratios of $L$-functions is one way to obtain very detailed formulae for lower order terms in the correlations of zeros, as demonstrated in \cite{kn:consna08} for the Riemann zeta function.  So, the conclusion is that  if you follow a number theoretical method in random matrix theory then you arrive at a formula for the $n$-level density that is not as clean as (\ref{eq:determinantal}) but is far more useful for comparison to number theory.  Almost simultaneously with \cite{kn:consna14}, Entin gave another method \cite{kn:entin14} for reproducing Rudnick and Sarnak's work without their complicated combinatorics.  He uses a comparison with Artin-Schreier function field $L$-functions. 

In a recent paper \cite{kn:massna16} we derive an alternative formula for the $n$-level densities of eigenvalues from matrices in $SO(2N)$ and $USp(2N)$ that is calculated, as in \cite{kn:consna08}, via averages of ratios of characteristic polynomials.  In the present paper we will demonstrate that armed with this alternative formula, it is relatively straightforward to reproduce the work done by Rubinstein \cite{kn:rub01} for $\sum_{i=1}^n |u_i|<1$ and Gao \cite{kn:gao14}, Levinson and Miller \cite{kn:levmil13} and Entin, Roditty-Gershon and Rudnick \cite{kn:entrodrud} for $\sum_{i=1}^n |u_i|<2$ to match their number theoretical formulae with random matrix theory.  We will also write down a formula for the limiting random matrix $n$-level density for support in $\sum_{i=1}^n |u_i|<3$ that closely follows the form of Gao's expression.  As far as we know there are no rigourous number theoretical results for this range of support and we hope that presenting the result explicitly in a helpful form may bring nearer the day when such a rigourous result is obtained.  

We will see that that  the $n$-level density of the zeros of quadratic Dirichlet $L$-functions in the limit of large conductor does not contain arithmetic information specific to the family, so the proof given here that it coincides with random matrix theory also holds for any other symplectic family if the $n$-level density is worked into the form of (\ref{eq:rubinstein}) or (\ref{eq:gaotheorem}).  The $n$-level densities of orthogonal families have almost identical structure, but with a few plus and minus signs exchanged, so no new results would be needed to compare the $n$-level density of an orthogonal family with the orthogonal formula in \cite{kn:massna16}.

We restate here the alternative $n$-level density formula from Section 7.2 of \cite{kn:massna16}.   So that we can extend our contours of integration to infinity in analogy with the number theory case, we extend the eigenangles of the matrices periodically (note that this notation is slightly different from that in \cite{kn:massna16} but the idea is the same).  So, if we have a matrix in $USp(2N)$ with eigenvalues $e^{\pm i\theta_1}, \ldots,e^{\pm i\theta_N}$, we can define for integer $k$ and for  $r=\pm1, \ldots,\pm N$ the sequence of angles
\begin{equation}
\theta_{r}+2k\pi=\left\{\begin{array}{cl}  \theta_{r+2kN} & {\rm if\;} r>0{\rm \;and\;} k\geq0\\ \theta_{r+2kN-1} &{\rm if \;} r>0 {\rm \; and\;} k<0 \\ \theta_{r+2kN+1}&{\rm if\;} r<0 {\rm \;and\;} k>0 \\ \theta_{r+2kN} & {\rm if\;} r<0{\rm\;and\;} k\leq 0 \end{array}\right.
\end{equation}

\begin{theorem}(Mason and Snaith \cite{kn:massna16}) \label{theo:massna}
For a matrix in $USp(2N)$ with eigenvalues $e^{\pm i\theta_1}, \ldots,e^{\pm i\theta_N}$ extended periodically as above and a test function that is holomorphic and decays rapidly in each variable in horizontal strips, we have
\begin{align}
\begin{split}
 &\int_{USp\left(2N\right)} {\sum_{\substack{j_1, \cdots, j_n =-\infty\\j_1,\ldots,j_n\neq 0}}^\infty {F\left(\theta_{j_1}, \cdots,
\theta_{j_n}\right) d\mu_{{\rm Haar}}} } \\
&= \frac{1}{\left(2 \pi i \right)^n} \sum_{K\cup L \cup M = \left\{1, \cdots,
n\right\}}{ (2N) ^{\left|M\right|}} \\
& \qquad \times \int_{\left(\delta\right)^{|K|}} \int_{(-\delta)^{|L|}}\int_{\left(0\right)^{|M|}}{J_{USp(2N)}^{*}\left(z_K\cup -z_L \right)F\left(iz_1, \cdots,
iz_n\right)dz_1 \cdots d z_n}. \label{eq:Uworking}
\end{split}
\end{align}
where
\begin{align}
\begin{split}\label{def:UJstar}
J_{USp(2N)}^*(A) &= \sum_{D\subseteq A}{e^{-2N \sum\limits_{d \in D}{d}}(-1)^{|D|} \sqrt{\frac{Z(D,D)Z(D^-,D^-)Y(D^-)}{Y(D)Z^{\dag}(D^{-},D)^2}}}\\
& \qquad \times \sum_{\substack{A/D = W_1 \cup \cdots \cup W_R \\ |W_r| \leq 2}}{\prod_{r=1}^{R}{H_D(W_r)}}
\end{split}
\end{align}
and
\begin{align}
H_D(W_r)=
\begin{cases}
\begin{aligned}
&\sum_{d \in D}{\left(\frac{z^{'}}{z}\left( \alpha-d\right)-\frac{z^{'}}{z}\left(\alpha+d\right)\right)}\\
& \qquad + \frac{z^{'}}{z}\left(2 \alpha\right)
\end{aligned} & W_r = \{\alpha\}\\
\begin{aligned}
\left(\frac{z^{'}}{z}\right)^{'}\left(\alpha + \beta\right)\end{aligned} & W_r = \{\alpha, \beta\}\\
1& W_r=\emptyset
\end{cases}
\end{align}
with
\begin{align}
z(x) &= \frac{1}{1-e^{-x}},\\
Y(A)&= \prod_{\alpha \in A}{z(2\alpha)}\label{def:orthoY},\\
Z(A,B) &= \prod_{\substack{\alpha \in A \\ \beta \in B}}{z(\alpha + \beta)}\label{def:orthoZ}.
\end{align}
The dagger on  $Z$ adds a restriction that a factor $z(x)$ is omitted if its argument is zero. 

Here $A$ is a set of complex $\alpha$.  $D^{-}=\{-\alpha: \alpha \in D\}$, $A \backslash D=\{\alpha \in A, \alpha \not\in D\}$ and the notation $K\cup L \cup M = \left\{1, \cdots,
n\right\}$ and $A/D = W_1 \cup \cdots \cup W_R$ means a sum over disjoint subsets.   $|M|$ is the size of set $M$.   The contour of integration denoted $(\delta)^{|K|}$ means all the $z$ variables with index in $K$  are integrated on the vertical line with real part $\delta$.    In (\ref{eq:Uworking}) $J_{USp(2N)}^*$ is evaluated at $z$s with index in the set $K$ or $L$. The $z$s with index in $L$ appear with a negative sign.  
\end{theorem}

The test function that we will be working with in this paper satisfies the symmetry property
\begin{equation}
F(x_1,\ldots,x_n)=F(\pm x_1,\ldots,\pm x_n),
\end{equation}
which allows us to write (\ref{eq:Uworking}) as
\begin{align}
\begin{split}
 &\int_{USp\left(2N\right)} {\sum_{\substack{j_1, \cdots, j_n =-\infty\\j_1,\ldots,j_n\neq 0}}^\infty {F\left(\theta_{j_1}, \cdots,
\theta_{j_n}\right) d\mu_{{\rm Haar}}} } \\
&= \frac{1}{\left(2 \pi i \right)^n} \sum_{Q \cup M = \left\{1, \cdots,
n\right\}}{ (2N) ^{\left|M\right|}} \\
& \qquad \times \int_{\left(\delta\right)^{|Q|}} \int_{\left(0\right)^{|M|}}{2^{|Q|}J_{USp(2N)}^{*}\left(z_Q \right)F\left(iz_1, \cdots,
iz_n\right)dz_1 \cdots d z_n}. \label{eq:Uworking2}
\end{split}
\end{align}
To obtain this we make a change of variables $z_\ell\rightarrow -z_\ell$ in (\ref{eq:Uworking}) and then combine the sets $K$ and $L$ into the set $Q$.  However, since any $j\in Q$ could have originated in either $K$ {\em or} $L$, we acquire a factor of $2^{|Q|}$.  

As in the unitary case \cite{kn:consna14} the formulae in Theorem \ref{theo:massna} simplify dramatically when the support of the Fourier transform of the test function is restricted. Assume a test function of the form $f_1(x_1)\cdots f_n(x_n) $ so its Fourier transform has the form $\prod_{i=1}^n \hat{f_i}(u_i)$, with the definition of $\hat{f}$ at (\ref{eq:fourier}).  Assume further that $\prod_{i=1}^n \hat{f_i}(u_i)\neq 0$ only if $\sum_{j=1}^n |u_i|<q$ for some integer $q$. 

The definition of the inverse Fourier transform allows us to write
\begin{equation}
f\Big(\tfrac{N}{\pi}iz\Big) = \int_{-\infty}^\infty \hat{f}(u) e^{2Nzu}du.
\end{equation}
Thus if we are considering the $n$-level density of eigenvalues scaled by their mean density 
\begin{equation}
\int_{USp\left(2N\right)} {\sum_{\substack{j_1, \cdots, j_n =-\infty\\j_1,\ldots,j_n\neq 0}}^\infty {f_1\Big(\tfrac{N}{\pi}\theta_{j_1}\Big) \cdots f_n\Big(\tfrac{N}{\pi}\theta_{j_n}\Big)
d\mu_{{\rm Haar}}} }
\end{equation}
then in the contour integral on the right hand side of (\ref{eq:Uworking2}) there will be a factor with magnitude
\begin{align}
\left|\exp\left(2Nz_1u_1+ \cdots +2Nz_nu_n\right)\right| = e^{2N \delta (q-\varepsilon)},\label{USpfactor1}
\end{align}
for some $\varepsilon>0$.
If we consider only sets $D$ in (\ref{def:UJstar}) where $|D|\geq q$, we can see the exponential term $e^{-2N\sum_{d\in D} d}$ in $J_{USp(2N)}^{*}(A)$ is bounded by
\begin{align}
\left|e^{-2N \sum_{d \in D}{d}}\right| \leq e^{-2N\delta q} \label{USpfactor2}.
\end{align}
Hence the product of these two factors in \ref{USpfactor1} and \ref{USpfactor2} tends to zero as $\delta\rightarrow \infty$ as we move the contours of integration off to the right and only terms with $|D|<q$ survive.

Define
\begin{align}
\begin{split}\label{def:UJq}
J_{USp(2N),q}^*(A) &= \sum_{\substack{D\subseteq A\\ |D| < q}}{e^{-2N \sum\limits_{d \in D}{d}}(-1)^{|D|} \sqrt{\frac{Z(D,D)Z(D^-,D^-)Y(D^-)}{Y(D)Z^{\dag}(D^{-},D)^2}}}\\
& \qquad \times \sum_{\substack{A/D = W_1 \cup \cdots \cup W_R \\ |W_r| \leq 2}}{\prod_{r=1}^{R}{H_D(W_r)}},
\end{split}
\end{align}
for use when the support of the Fourier transform of the test function is in $\sum_{j=1}^n |u_j|<q$.

\section{Results}\label{sect:results}

The first purpose of this paper is to use the form of the $n$-level density in \cite{kn:massna16} to demonstrate for $q=1$ or $q=2$ that for test functions $f_1, \ldots, f_n$ with the product of their Fourier transforms $\prod_{i=1}^n \hat{f_i}(u_i)$ having support in $\sum_{i=1}^n |u_i|<q$, 
\begin{eqnarray}\label{eq:result}
&&\lim_{X\rightarrow \infty} \frac{1}{|D(X)|}  \sum_{d\in D(X)} \sum_{\substack{j_1,\ldots,j_n=-\infty\\j_1,\ldots,j_n\neq 0}}^\infty f_1\Big(L\gamma_d^{(j_1)} \Big)\cdots f_n\Big(L\gamma_d^{(j_n)}\Big)\nonumber \\
&& = \lim_{N\rightarrow \infty}  \int_{USp(2N)} \sum _{\substack{j_1,\ldots,j_n=-\infty \\j_1,\ldots,j_n\neq 0}}^\infty f_1\Big(\frac{N}{\pi} \theta_{j_1}\Big) \cdots f_n\Big(\frac{N}{\pi} \theta_{j_n}\Big) d\mu_{{\rm Haar}},
\end{eqnarray}
where the eigenangles $\theta_{j_i}$ are counted as described at Theorem \ref{theo:massna} and 
\begin{equation}
L=\frac{\log X}{2\pi},
\end{equation}
and 
\begin{equation}
1/2+i\gamma_d^{(j)}, \;\; j=\pm1,\pm2,\ldots,
\end{equation}
are the nontrivial zeros of $L(s,\chi_d)$ with
\begin{equation}
0\leq \Re \gamma_d^{(1)} \leq \Re \gamma_d^{(2)}\leq \Re \gamma_d^{(3)}\cdots
\end{equation}
and 
\begin{equation}
\gamma_d^{(-j)}=-\gamma_d^{(j)}.
\end{equation}
Note that this result is not new in itself but the point is to demonstrate that the use of  (\ref{eq:Uworking2}) makes matching the left hand side with the right hand side of  (\ref{eq:result}) much easier than if the determinantal form is used.  In Section \ref{sect:support1} we rederive the result of Rubinstein \cite{kn:rub01} for $\sum_{i=1}^n |u_i|<1$ and in Section \ref{sect:support2} we reproduce the result of Gao \cite{kn:gao14}, Levinson and Miller \cite{kn:levmil13} and Entin, Roditty-Gershon and Rudnick \cite{kn:entrodrud} for $\sum_{i=1}^n |u_i|<2$.  

Secondly, in Section \ref{sect:support3} we move to support $\sum_{i=1}^n |u_i|<q$ with $q=3$ and write  (\ref{eq:Uworking2}) in the style of Gao.  We note that no significant new ideas are needed for this beyond those used for $q=1$ and $q=2$, so on the random matrix side it seems that there is no fundamental barrier to writing down a formula like Gao's for any range of support - one just has to work with more involved formulae.  However, extending the support beyond $q=2$ in a rigourous number theoretical context seems to require completely new ideas and it is our hope that by writing down explicitly the final answer for $q=3$ in Section \ref{sect:support3} we may bring closer the day when such a calculation can be accomplished rigourously. 

We note that the calculations are carried out here for the quadratic family of Dirichlet $L$-functions, but no new results would be needed to compare the $n$-level density of any family of orthogonal or symplectic $L$-functions as long as the $n$-level density of the zeros was written in the form of Rubinstein or Gao by applying the explicit formula in a similar way. 

To prove our result we need the following key lemmas.  In all of the following we have the definition 
\begin{equation} \label{eq:z}
z(x)=(1-e^{-x})^{-1}, 
\end{equation}
and we assume the following about the test functions:
\begin{equation}\label{eq:testfunction}
\begin{array}{l} {\rm a)\; the\;test\;function\;has \;the\;form\;of\;a\;product\;} f_1(x_1)f_2(x_2)\cdots f_n(x_n);\\
{\rm b)\;each\;component\;is\;even,\;} f_j(x)=f_j(-x);\\
{\rm c)\;}f_j(x_j){\rm\;has\;no\;poles;}\\
{\rm d)\;} f_j(x_j){\rm \; is\; smooth;}\\
{\rm e)\;} f_j(x_j){\rm \; decays \;quickly\;in\;horizontal\;stips.}\end{array}
\end{equation}
The integrals are on the vertical lines with real part $\delta>0$ or $-\delta$, as indicated. 

The first three relatively simple lemmas are all that is needed if the support of the Fourier transform of the test function is restricted to $\sum_{i=1}^n |u_i|<1$.
\begin{lemma} \label{lem:1}
For test functions with the properties at (\ref{eq:testfunction}),
\begin{eqnarray}
&&\lim_{N\rightarrow \infty} \frac{4}{(2\pi i)^2} \int_{(\delta)} \int_{(\delta)} \left( \frac{z'}{z}\right)' (z_1+z_2)f_1(\tfrac{N}{\pi} iz_1) f_2(\tfrac{N}{\pi} iz_2) dz_1 dz_2 \nonumber \\
&& \qquad = 2\int_{-\infty}^{\infty} \hat{f_1}(u)\hat{f_2}(u) |u| du.
\end{eqnarray}
\end{lemma}

\begin{lemma}\label{lem:2}For test functions with the properties at (\ref{eq:testfunction}),
\begin{eqnarray}
&&
\lim_{N\rightarrow \infty} \frac{2}{2\pi i} \int_{(\delta)} \frac{z'}{z}(2z) f(\tfrac{N}{\pi} i z) dz \nonumber \\
&&\qquad = -\frac{1}{2} \int_{-\infty}^{\infty} \hat{f_1}(u) du.
\end{eqnarray}
\end{lemma}

\begin{lemma}\label{lem:3}For test functions with the properties at (\ref{eq:testfunction}),
\begin{equation}
\frac{2N}{2\pi i} \int_{(-\delta)} f(\tfrac{N}{\pi} i z) dz = \int_{-\infty} ^{\infty} f(y)dy
\end{equation}
\end{lemma}

The fourth lemma is needed when support is extended to $\sum_{i=1}^n |u_i|<2$.
\begin{lemma}\label{lem:4}
For  disjoint sets $A\cup B=\{1,\ldots,m\}$ with $|B|\geq 1$ and $B\backslash \{d\}$ denoting the set $B$ excluding the element $d$, and for test functions with the properties at (\ref{eq:testfunction}),
\begin{eqnarray}
&&\lim_{N\rightarrow \infty} \sum_{d\in B} \frac{2}{2\pi i} \int_{(\delta)}  -e^{-2Nz_d} z(-2z_d) \left( \prod_{i\in A} \frac{-2}{2 \pi i} \int_{(\delta)}\frac{z'}{z}(z_i+z_d) f_i\Big(\tfrac{N}{\pi} i z_i\Big)  dz_i\right) \nonumber \\
&&\times \left( \prod_{j\in B\backslash \{d\}} \frac{2}{2\pi i} \int_{(\delta)} \frac{z'}{z} (z_j-z_d)f_j\Big(\tfrac{N}{\pi} i z_j\Big) dz_j \right) f_d\Big(\tfrac{N}{\pi} i z_d\Big) dz_d \nonumber\\
&&=\frac{-1}{2}2^{|A\cup B|}(-1)^{|B|} \int_{\substack{(\mathbb{R}\geq 0)^{A\cup B} \\ \sum_{i\in A}u_i\leq (\sum_{i\in B} u_i)-1}} \prod_{i\in A\cup B} \hat{f_i}(u_i) \prod_{i\in A\cup B} du_i.
\end{eqnarray}
\end{lemma}

The final lemma is needed only when support is extended to $\sum_{i=1}^n |u_i|<3$.

\begin{lemma} \label{lem:5}
For  disjoint sets $A_1\cup B_1\cup A_2 \cup B_2=\{1,\ldots,m\}$ with $|B_i|\geq 1$ and $B\backslash \{d\}$ denoting the set $B$ excluding the element $d$, and with test functions having the properties in (\ref{eq:testfunction}),
\begin{eqnarray}
&& \lim _{N\rightarrow \infty}\sum_{d\in B_1}\sum_{g \in B_2}\frac{4}{(2\pi i)^2}\int_{(\delta)}\int_{(\delta)} \left(e^{-2N(z_d+z_g)}\frac{z(z_d+z_g) z(-z_d-z_g)z(-2z_d)z(-2z_g)} {z(z_d-z_g) z(z_g-z_d)}\right) \nonumber\\
&&\qquad\qquad \qquad\qquad\times \left( \prod_{i\in A_1} \frac{2}{2\pi i}\int_{(\delta)}-\frac{z'}{z}(z_i+z_d)f_i\Big( \tfrac{N}{\pi}iz_i\Big) dz_i\right) \nonumber \\
&&\qquad\qquad\qquad\qquad\times\left(\prod_{j\in B_1\backslash\{d\}}\frac{2}{2\pi i} \int_{(\delta)} \frac{z'}{z}(z_j-z_d)f_j\Big( \tfrac{N}{\pi}iz_j\Big) dz_j \right)\nonumber \\
&&\qquad\qquad\qquad\qquad \times \left( \prod_{i\in A_2} \frac{2}{2\pi i}\int_{(\delta)}-\frac{z'}{z}(z_i+z_g)f_i\Big( \tfrac{N}{\pi}iz_i\Big) dz_i\right) \nonumber \\
&&\qquad\qquad\qquad\qquad\times\left(\prod_{j\in B_2\backslash\{g\}}\frac{2}{2\pi i} \int_{(\delta)} \frac{z'}{z}(z_j-z_g)f_j\Big( \tfrac{N}{\pi}iz_j\Big) dz_j \right)\nonumber \\
&&\qquad\qquad\qquad\qquad\times f_d\Big( \tfrac{N}{\pi}iz_d\Big) f_g\Big(\tfrac{N}{\pi} iz_g\Big)dz_d dz_g\left. \right) \nonumber \\
&&=2^{m}(-1)^{|B_1\cup B_2|} \int_{\substack{(\mathbb{R}\geq 0)^{|A_1\cup B_1\cup A_2 \cup B_2|} \\ \sum_{i\in A_1}u_i\leq (\sum_{i\in B_1} u_i)-1\\ \sum_{i\in A_2}u_i\leq (\sum_{i\in B_2} u_i)-1}}\left( \frac{1}{4}-\left[\left(-\sum_{i\in A_2}u_i+\sum_{j\in B_2} u_j-1\right)\right.\right.\nonumber \\
&&\qquad\qquad \times\left.\left. \delta\left(-\sum_{i\in A_1\cup B_2}u_i+\sum_{j\in B_1\cup A_2} u_j\right)\right]\right) \prod_{i\in \{1,\ldots,m\}} \hat{f_i}(u_i) \prod_{i\in \{1,\ldots,m\}} du_i.
\end{eqnarray}
\end{lemma}

\section{Proof of key lemmas}

We will now present the proof of the key lemmas.  The first three are fairly straightforward, but we will prove each for completeness. It is useful to note the following behaviour for $x$ of small magnitude:

\begin{eqnarray}
z(x)&:=&\frac{1}{1-e^{-x}} \sim \frac{1}{x}+ O(1),\label{eq:zasymp}\\
\frac{z'}{z}(x)&\sim& -\frac{1}{x} +O(1),\label{eq:zlogderivasymp}\\
\left(\frac{z'}{z}\right)'(x) & \sim& \frac{1}{x^2}+O(x^{-1}).\label{eq:zdoublederivasymp}
\end{eqnarray}

\begin{proof}[Proof of Lemma \ref{lem:1}]
We start with the left hand side of the statement of Lemma \ref{lem:1} and perform a change of variables:
\begin{eqnarray}
&&\lim_{N\rightarrow \infty} \frac{4}{(2\pi i)^2} \int_{(\delta)} \int_{(\delta)} \left( \frac{z'}{z}\right)' (z_1+z_2)f_1(\tfrac{N}{\pi} iz_1) f_2(\tfrac{N}{\pi} iz_2) dz_1 dz_2 \nonumber \\
&&=\lim_{N\rightarrow \infty} \frac{4}{(2N i)^2} \int_{(\tfrac{N\delta}{\pi})} \int_{(\tfrac{N\delta}{\pi})} \left( \frac{z'}{z}\right)' (\tfrac{\pi}{N}z_1+\tfrac{\pi}{N}z_2)f_1( iz_1) f_2( iz_2) dz_1 dz_2.
\end{eqnarray}
Since the $f_j$ have no poles in the right half-plane and $\left( \frac{z'}{z}\right)' (x)$ has a pole only at zero, the contours can be shifted back to $(\delta)$ for convenience.  Here and elsewhere we can ignore horizontal segments of the contour of integration when shifting contours as these will be negligible if the $f_j$ are Schwartz functions, or otherwise decay sufficiently fast in horizontal strips (that is, when the absolute value of the real part of the argument gets large).  

Now we apply the large $N$ limit and use (\ref{eq:zdoublederivasymp}) to obtain
\begin{eqnarray}
&&\frac{4}{(2\pi i)^2} \int_{(\delta)}\int_{(\delta)}  \frac{1}{(z_1+z_2)^2} f_1(iz_1)f_2(iz_2)dz_1 dz_2\\
&&=\int_{-\infty}^{\infty} \int_{-\infty}^{\infty}\hat{f_1}(u_1) \hat{f_2}(u_2) \frac{4}{(2\pi i)^2} \int_{(\delta)} \int_{(\delta)} \frac{1}{(z_1+z_2)^2} e^{2\pi z_1u_1+2\pi z_2u_2} dz_1dz_2du_1du_2.\nonumber
\end{eqnarray}
where we have inserted the definition of the Fourier transform of the $f_j$s from (\ref{eq:fourier}). 

We now notice that if either of the variables $u_1$ or $u_2$ is negative, then we could close the corresponding $z$ integral in the right half plane with a contour on the vertical line with real part $R$ and as $R\rightarrow \infty$ the contribution of closing the contour would go to zero due to the factor $e^{2\pi z_ju_j}$.  As there are no poles of the integrand in the right half plane, this means that if either of the variables $u_1$ or $u_2$ is negative then the whole integral is zero.

Thus we are left looking at
\begin{equation}\label{eq:pospos}
\int_{0}^{\infty} \int_0^\infty  \hat{f_1}(u_1) \hat{f_2}(u_2) \frac{4}{(2\pi i)^2} \int_{(\delta)} \int_{(\delta)} \frac{1}{(z_1+z_2)^2} e^{2\pi z_1u_1+2\pi z_2u_2} dz_1dz_2du_1du_2.
\end{equation}
As the $u_j$s are positive, closing the $z_2$ contour to the far left will result in negligible contribution from that contour of integration.  However, the contour will enclose the pole at $z_2=-z_1$, resulting in a residue
\begin{equation}
\Res_{z_2=-z_1} \left( \frac{1}{(z_1+z_2)^2} e^{2\pi z_1u_1+2\pi z_2u_2}\right)=2\pi u_2 e^{2\pi z_1(u_1-u_2)}.
\end{equation}
Thus (\ref{eq:pospos})  becomes
\begin{eqnarray}
&&\int_{0}^{\infty} \int_0^\infty  \hat{f_1}(u_1) \hat{f_2}(u_2) \frac{4}{2\pi i} \int_{(\delta)} 2\pi u_2 e^{2\pi z_1(u_1-u_2)} dz_1du_1du_2 \nonumber \\
&&=4\int_{0}^{\infty} \int_0^\infty  \hat{f_1}(u_1) \hat{f_2}(u_2) u_2 \delta(u_1-u_2) du_1 du_2,
\end{eqnarray}
where we have shifted the $z_1$ contour onto the imaginary axis, made a change of variables $z_1\rightarrow -iz_1$, and then used the expression for a delta-function
\begin{equation}\label{eq:delta}
\delta(x)=\int_{-\infty}^{\infty} e^{-2\pi i y x } dy.
\end{equation}
Thus we have that (\ref{eq:pospos}) is
\begin{eqnarray}
4\int_0^\infty \hat{f_1}(u)\hat{f_2}(u) u du&=& 2\int_{-\infty}^\infty \hat{f_1}(u)\hat{f_2}(u) |u| du, 
\end{eqnarray}
using the property that $\hat{f_j}$ are even functions. 
\end{proof}

\begin{proof}[Proof of Lemma \ref{lem:2}]
We start with the left hand side of Lemma \ref{lem:2}, perform the same change of variables and the same procedure for taking the $N\rightarrow \infty$ limit, although this time using (\ref{eq:zlogderivasymp}), as in the proof of Lemma \ref{lem:1}.  This leads us to
\begin{eqnarray}
&&
\lim_{N\rightarrow \infty} \frac{2}{2\pi i} \int_{(\delta)} \frac{z'}{z}(2z) f(\tfrac{N}{\pi} i z) dz \nonumber \\
&&=\frac{-1}{2\pi i} \int_{(\delta)} \frac{1}{z} f(iz) dz.
\end{eqnarray}
We now write $f$ in terms of its Fourier transform (\ref{eq:fourier}).  So, the above equation equals
\begin{equation}
\frac{-1}{2\pi i}\int_{-\infty}^\infty \hat{f}(u) \int_{(\delta)}\frac{1}{z} e^{2\pi zu} dz du.
\end{equation}
If $u<0$ then the $z$ contour could be closed far to the right and would give a zero contribution. So, we look at
\begin{equation}
\frac{-1}{2\pi i}\int_{0}^\infty \hat{f}(u) \int_{(\delta)}\frac{1}{z} e^{2\pi zu} dz du,
\end{equation}
and close the $z$ contour far to the left, picking up a pole at $z=0$ with residue 1. 

Thus our final result is
\begin{equation}
-\int_0^\infty \hat{f}(u) du, 
\end{equation}
which, due to $\hat{f}$ being even, equals
\begin{equation}
-\frac{1}{2} \int_{-\infty}^{\infty} \hat{f}(u) du.
\end{equation}
\end{proof}

\begin{proof}[Proof of Lemma \ref{lem:3}]
This is very straightforward indeed. We merely perform a change of variable $z=- \pi i y/N$ to see that
\begin{equation}
\frac{2N}{2\pi i} \int_{(-\delta)} f(\tfrac{N}{\pi} iz) dz = \int_{-\infty}^{\infty} f(y)dy.
\end{equation}
\end{proof}

The three lemmas above are all that are needed to compare with Rubinstein's result with support on $\sum_{i=1}^n |u_i|<1$ in Section \ref{sect:support1}.

\begin{proof}[Proof of Lemma \ref{lem:4}]
We write out the left hand side of the equality in the lemma
\begin{eqnarray}\label{eq:lem4a}
&&\lim_{N\rightarrow \infty} \sum_{d\in B} \frac{2}{2\pi i} \int_{(\delta)}  -e^{-2Nz_d} z(-2z_d) \left( \prod_{i\in A} \frac{-2}{2 \pi i} \int_{(\delta)}\frac{z'}{z}(z_i+z_d) f_i\Big(\tfrac{N}{\pi} i z_i\Big)  dz_i\right) \nonumber \\
&&\times \left( \prod_{j\in B\backslash \{d\}} \frac{2}{2\pi i} \int_{(\delta)} \frac{z'}{z} (z_j-z_d)f_j\Big(\tfrac{N}{\pi} i z_j\Big) dz_j \right) f_d\Big(\tfrac{N}{\pi} i z_d\Big) dz_d.
\end{eqnarray}
We perform the same steps as at the start of the proof of Lemma \ref{lem:1}.  That is, we change variables $z_i\rightarrow \tfrac{\pi}{N}z_i$ and then for large $N$ approximate the $z$ functions asymptotically using (\ref{eq:zasymp}) and (\ref{eq:zlogderivasymp}).   Thus we arrive at
\begin{eqnarray}\label{eq:lem4aa}
&&\sum_{d\in B}\frac{2}{2\pi i}  \int_{(\delta)} \left(-e^{-2\pi z_d}\frac{-1}{2z_d} \right) \nonumber \\
&&\times \left( \prod_{i\in A} \frac{2}{2 \pi i} \int_{(\delta)} \frac{1}{z_i+z_d}f_i(iz_i) dz_i\right) \left( \prod_{j\in B\backslash \{d\}} \frac{2}{2\pi i} \int_{(\delta)} \frac{-1}{z_j-z_d } f_j(iz_j)dz_j\right) f_d(iz_d)dz_d .
\end{eqnarray}
Note that after the change of variables we have brought all the contours back to their original positions.  This does not have to be done carefully because although there appear to be poles when $z_j=z_d$ for each $j\in B$, in fact they all cancel out.  Consider, for example, $z_1$ and $z_2$ when $1,2\in B$.  When $d=1$, 
\begin{eqnarray}
&&\Res_{z_1=z_2} \left( -e^{-2\pi z_1} \frac{-f_1(iz_1)}{2z_1} \times \prod_{i\in A} \frac{f_i(iz_i)}{z_i+z_1} \times \prod_{j\in B\backslash\{1,2\}}\frac{-f_j(iz_j)}{z_j-z_1} \times \frac{-f_2(iz_2)}{z_2-z_1} \right)\nonumber \\
&&\qquad= -e^{-2\pi z_2} \frac{-f_1(iz_2)}{2z_2} \times \prod_{i\in A} \frac{f_i(iz_i)}{z_i+z_2} \times \prod_{j\in B\backslash\{1,2\}}\frac{-f_j(iz_j)}{z_j-z_2} f_2(iz_2) ,
\end{eqnarray}
whereas when $d=2$,
\begin{eqnarray}
&&\Res_{z_1=z_2} \left( -e^{-2\pi z_2} \frac{-f_2(iz_2)}{2z_2} \times \prod_{i\in A} \frac{f_i(iz_i)}{z_i+z_2} \times \prod_{j\in B\backslash\{1,2\}}\frac{-f_j(iz_j)}{z_j-z_2} \times \frac{-f_1(iz_1)}{z_1-z_2} \right)\nonumber \\
&&\qquad= e^{-2\pi z_2} \frac{-f_2(iz_2)}{2z_2} \times \prod_{i\in A} \frac{f_i(iz_i)}{z_i+z_2} \times \prod_{j\in B\backslash\{1,2\}}\frac{-f_j(iz_j)}{z_j-z_2}f_1(iz_2) .
\end{eqnarray}
Thus these two residues cancel each other out. 

The fact that there is no pole for any $z$ in the right half plane also means that we can spread the contours slightly.  We will define a sequence $0<\delta_1<\delta_2<\cdots<\delta_m$ and $z_i$ will be integrated on the vertical line at $\delta_i$.

We now substitute each $f$ with its definition in terms of its Fourier transform.  After these manipulations, (\ref{eq:lem4a}) equals
\begin{eqnarray}\label{eq:lem4b}
&&\sum_{d\in B}\frac{2}{2\pi i} \int_{-\infty}^\infty \cdots \int_{-\infty}^{\infty} \left( \prod_{i\in A} \hat{f_i}(u_i)\right) \left( \prod_{j\in B}\hat{f_j}(u_j) \right) \int_{(\delta_d)} \left(-e^{2\pi z_d(u_d-1)}\frac{-1}{2z_d} \right) \nonumber \\
&&\times \left( \prod_{i\in A} \frac{2}{2 \pi i} \int_{(\delta_i)} \frac{e^{2\pi u_iz_i}}{z_i+z_d} dz_i\right) \left( \prod_{j\in B\backslash \{d\}} \frac{2}{2\pi i} \int_{(\delta_j)} \frac{-e^{2\pi u_jz_j}}{z_j-z_d } dz_j\right) dz_d du_1 \cdots du_m.
\end{eqnarray}

We consider first the term in the sum over $d$ where $d<j, \forall j\in B\backslash\{d\}$.  So, $d$ is the smallest element in $B$. For convenience we will label the elements of $B$ in ascending order: $j_1< j_2<\ldots< j_{|B|}$.  The integrals over the $u$ variables run from $-\infty$ to $\infty$ but we note that the behaviour of the exponentials  $\exp(2\pi u_i z_i)$ and $\exp(2\pi u_jz_j)$ depends on whether the $u$ variable is positive or negative.  Thus we consider separately the case where each $u$ is greater than or less than zero.  We observe that if any $u_j<0$,  $j\in A\cup B\backslash\{j_1\}$, then for the term corresponding to  $d=j_1$ in (\ref{eq:lem4b}) the integral over the corresponding $z_j$, $j\in A,$ or $z_j$, $j\in B\backslash\{j_1\},$ can be closed to the right and the contribution incurred by closing the contour will be zero since $e^{2\pi u_j z_j}$ goes to zero for $z_j$ with large real part and $u_j<0$.  We also note that in closing the contour to the right, we don't enclose any poles (because the $z_{j_1}$ contour is to the left of all the other contours and so no pole $z_j=z_{j_1}$ would be encountered).   See Figure \ref{fig:contours} for a sketch of how the contours are arranged (it shows the general case where $d$ is not necessarily $j_1$).  Thus the contribution to the term corresponding to $d=j_1$ is
\begin{eqnarray}\label{eq:lem4c}
&&\frac{2}{2\pi i} \int_{-\infty}^\infty \; \int_{0}^\infty \cdots \int_{0}^{\infty} \left( \prod_{i\in A} \hat{f_i}(u_i)\right) \left( \prod_{j\in B}\hat{f_j}(u_j) \right) \int_{(\delta_{j_1})} \left(-e^{2\pi z_{j_1}(u_{j_1}-1)}\frac{-1}{2z_{j_1}} \right) \nonumber \\
&&\times \left( \prod_{i\in A} \frac{2}{2 \pi i} \int_{(\delta_i)} \frac{e^{2\pi u_iz_i}}{z_i+z_{j_1}} dz_i\right) \left( \prod_{j\in B\backslash \{j_1\}} \frac{2}{2\pi i} \int_{(\delta_j)} \frac{-e^{2\pi u_jz_j}}{z_j-z_{j_1} } dz_j\right) dz_{j_1} \left(\prod_{j\in A\cup B\backslash \{j_1\}} du_j \right)du_{j_1}.
\end{eqnarray}

Now we are left with $u_j>0$ for all $j\neq j_1$.  We can therefore close all the contours corresponding to $z_i$, $i\in A,$ and $z_j$, $j\in B\backslash\{j_1\},$ to the left.  This picks up residues at $z_i=-z_{j_1}$, for $i\in A$, and at $z_j=z_{j_1}$, for $j\in B\backslash \{{j_1}\}$.  The result is that (\ref{eq:lem4c}) equals
\begin{eqnarray}\label{eq:lem4c1}
&& = \frac{2^{|A\cup B|}}{2\pi i} \int_{-\infty}^\infty \; \int_{(\mathbb{R}\geq0)^{|B\backslash\{j_1\}|}}\left( \prod_{i\in A} \hat{f_i}(u_i)\right) \left( \prod_{j\in B}\hat{f_j}(u_j) \right) \int_{(\delta_{j_1})} \left(-e^{2\pi z_{j_1}(u_{j_1}-1)}\frac{-1}{2z_{j_1}} \right) \nonumber \\
&&\qquad
\times \left(\prod_{i\in A} e^{-2\pi u_iz_{j_1}}\right) \left( \prod_{j\in B\backslash \{j_1\}} -e^{2\pi u_jz_{j_1}}\right) dz_{j_1} \left(\prod_{j\in A\cup B\backslash \{j_1\}} du_j \right)du_{j_1}.
\end{eqnarray}

Note that if $\sum_{i\in A} u_i>(\sum_{j\in B} u_j)-1$ then $\exp(2\pi z_{j_1}(-\sum_{i\in A} u_i+(\sum_{j\in B} u_j)-1))$ will decay on a loop closing the $z_{j_1}$ contour far to the right, yielding no contribution at all.   However, if $\sum_{i\in A} u_i\leq(\sum_{j\in B} u_j)-1$ then we close the $z_{j_1}$ contour to the left, enclosing the residue at $z_{j_1}=0$.  Thus we have that (\ref{eq:lem4c1}) equals
\begin{eqnarray}\label{eq:lem4d}
&&  2^{|A\cup B|} (-1)^{|B|}\left(\frac{-1}{2}\right) \int_{-\infty}^\infty \; \int_{\substack{(\mathbb{R}\geq 0)^{|B\backslash \{j_1\}|}\\\sum_{i\in A} u_i\leq (\sum_{j\in B} u_j)-1} }\left( \prod_{i\in A\cup B} \hat{f_i}(u_i)\right)   \left(\prod_{j\in A\cup B\backslash \{j_1\}} du_j \right)du_{j_1}\nonumber \\
&&= 2^{|A\cup B|} (-1)^{|B|}\left(\frac{-1}{2}\right) \int_{0}^\infty \; \int_{\substack{(\mathbb{R}\geq 0)^{|B\backslash \{j_1\}|}\\\sum_{i\in A} u_i\leq (\sum_{j\in B} u_j)-1} }\left( \prod_{i\in A\cup B} \hat{f_i}(u_i)\right)   \left(\prod_{j\in A\cup B\backslash \{j_1\}} du_j \right)du_{j_1}\nonumber \\
&&+ 2^{|A\cup B|} (-1)^{|B|}\left(\frac{-1}{2}\right) \int_{-\infty}^0 \; \int_{\substack{(\mathbb{R}\geq 0)^{|B\backslash \{j_1\}|}\\\sum_{i\in A} u_i\leq (\sum_{j\in B} u_j)-1} }\left( \prod_{i\in A\cup B} \hat{f_i}(u_i)\right)   \left(\prod_{j\in A\cup B\backslash \{j_1\}} du_j \right)du_{j_1}\nonumber \\
&&= 2^{|A\cup B|} (-1)^{|B|}\left(\frac{-1}{2}\right)  \int_{\substack{(\mathbb{R}\geq 0)^{|B|}\\\sum_{i\in A} u_i\leq (\sum_{j\in B} u_j)-1} }\left( \prod_{i\in A\cup B} \hat{f_i}(u_i)\right)   \left(\prod_{j\in A\cup B} du_j \right)\nonumber \\
&&+ 2^{|A\cup B|} (-1)^{|B|}\left(\frac{-1}{2}\right) \int_{\substack{(\mathbb{R}\geq 0)^{|B|}\\(\sum_{i\in A} u_i)+u_{j_1}\leq (\sum_{j\in B\backslash\{j_1\}} u_j)-1} }\left( \prod_{i\in A\cup B} \hat{f_i}(u_i)\right)   \left(\prod_{j\in A\cup B} du_j \right).
\end{eqnarray}
For the final equality we have made a change of variable in the second integral $z_{j_1}\rightarrow -z_{j_1}$.  As $f_{j_1}$ is even, $\hat{f_{j_1}}$ is also even, but the sign of $u_{j_1}$ changes in the inequality controlling the domain of integration.  We notice that the first term after the final equality above is exactly the answer we want to prove the lemma.  We will now show that the second term cancels out when we sum over all $d\in B$. 

\begin{figure}[h!]
\caption{A sketch of the arrangement of the contours of integration employed in evaluating (\ref{eq:lem4b}). The sign of $u_j$ determines whether the contour of integration of the corresponding $z_j$ (marked with a $j$ on the sketch) can be closed to the right or the left due to $\exp(2\pi u_jz_j)$ in the integrand being small when the exponent is large and negative.  In the sketch the shaded region indicates the region enclosed by the $z_j$ contour.  A star indicates the situations when this contour encloses a pole.  Note that we have poles at $z_j=-z_{j_d}$, for $j\in A$, meaning that it makes no difference whether the $j$ contour lies to the left or the right of the $j_d$ contour.  For $j\in B\backslash \{{j_d}\}$ we have poles at $z_j=z_{j_d}$.  }
\label{fig:contours}
\includegraphics[width=7cm]{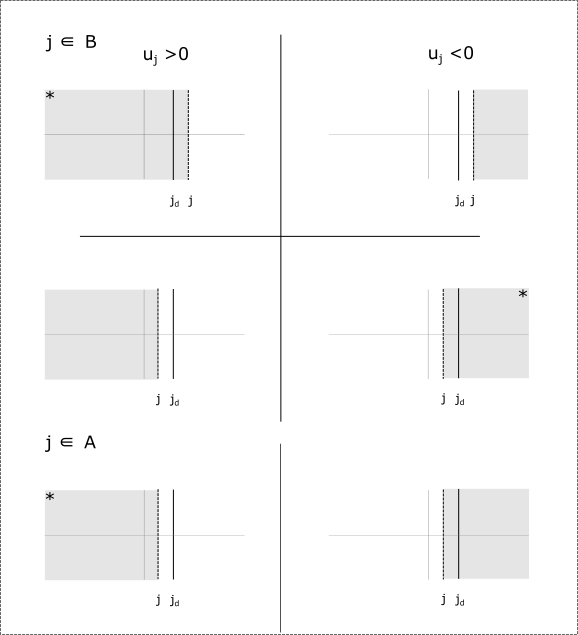}
\end{figure}

We have labelled the elements of $B$ as $j_1< j_2<\ldots< j_{|B|}$.  We now consider a term in (\ref{eq:lem4b}) corresponding to $d=j_d\neq j_1$.  So, $j_d$ is some element of $B$ other than the smallest element.  We will define $B_{<d}=B\cap \{1,\ldots,d-1\}$ and $B_{>d}=B\cap \{d+1,\ldots,m\}$.  We repeat the steps carried out for $d=j_1$, noting that each $z_i$, for $i\in A$, encounters a pole at $z_i=-z_{j_d}$ and each $z_j$, for $j\in B\backslash \{j_d\}$ encounters a pole at $z_j=z_{j_d}$.  However here we notice that for elements $j\in B_{<d}$, in the region where $u_j>0$, when we close the $z_j$-contour to the far left we enclose no poles.  For $j\in B_{<d}$  in the region where $u_j\leq 0$  when we close the contour far to the right,we pick up the pole at $z_{j}=z_{j_d}$.  This is due to the fact that before (\ref{eq:lem4b}) we spread the contours so that $0<\delta_1<\delta_2<\cdots<\delta_m$.  See Figure \ref{fig:contours} for a sketch of how the contours of integration are arranged. Thus we find that the term in (\ref{eq:lem4b}) corresponding to $d=j_d$ is
\begin{eqnarray}\label{eq:lem4e}
&& \frac{2^{|A\cup B|}}{2\pi i}\int_{(\mathbb{R}\leq0)^{|B_{<d}|}} \int_{-\infty}^\infty  \int_{(\mathbb{R}\geq0)^{|B_{>d}|}}  \left( \prod_{i\in A\cup B} \hat{f_i}(u_i)\right)  \int_{(\delta_{j_d})} \left(-e^{2\pi z_{j_d}(u_{j_d}-1)}\frac{-1}{2z_{j_d}} \right) \nonumber \\
&&\qquad
\times \left(\prod_{i\in A} e^{-2\pi u_iz_{j_d}}\right) \left( \prod_{j\in B_{<d}} e^{2\pi u_jz_{j_d}}\right)  \left( \prod_{j\in B_{>d}} -e^{2\pi u_jz_{j_d}}\right)dz_{j_d} \nonumber \\
&&\qquad \qquad \qquad \qquad\times \left(\prod_{j\in A\cup B_{>d}} du_j \right)du_{j_d}\left(\prod_{j\in A\cup B_{<d}} du_j \right)\nonumber \\
&&= \frac{2^{|A\cup B|}}{2\pi i}\int_{-\infty}^\infty  \int_{(\mathbb{R}\geq0)^{|B\backslash\{j_d\}|}}  \left( \prod_{i\in A\cup B} \hat{f_i}(u_i)\right)  \int_{(\delta_{j_d})} \left(-e^{2\pi z_{j_d}(u_{j_d}-1)}\frac{-1}{2z_{j_d}} \right) \nonumber \\
&&\qquad
\times \left(\prod_{i\in A} e^{-2\pi u_iz_{j_d}}\right) \left( \prod_{j\in B_{<d}} e^{-2\pi u_jz_{j_d}}\right)  \left( \prod_{j\in B_{>d}} -e^{2\pi u_jz_{j_d}}\right)dz_{j_d} \nonumber \\
&&\qquad \qquad \qquad \qquad\times \left(\prod_{j\in A\cup B\backslash\{j_d\}} du_j \right)du_{j_d} .
\end{eqnarray}
Note that for $j\in B_{<d}$ we close the contour to the right in a clockwise direction, resulting in the residue picking up an extra minus sign with respect to closing a contour to the left in a counter-clockwise direction (see the second line of (\ref{eq:lem4e})).  Now we just have the $z_{j_d}$ integral to perform.  We will close this contour to the far left and pick up the residue at $z_{j_d}=0$ under the condition that 
$\sum_{i\in A}u_i+\sum_{j\in B_{<d}} u_j \leq(\sum_{j \in B_{>d}} u_j)+u_{j_d}-1$.  So (\ref{eq:lem4e}) equals
\begin{eqnarray}
&& 2^{|A\cup B|} (-1)^{|B_{>d}|}\left(\frac{-1}{2}\right)\int_{-\infty}^\infty  \int_{\substack{(\mathbb{R}\geq0)^{|B\backslash\{j_d\}|}\\\sum_{i\in A}u_i+\sum_{j\in B_{<d}} u_j \leq(\sum_{j \in B_{>d}} u_j)+u_{j_d}-1 } } \left( \prod_{i\in A\cup B} \hat{f_i}(u_i)\right)  \nonumber \\
&&\qquad \qquad \qquad \qquad\times \left(\prod_{j\in A\cup B\backslash\{j_d\}} du_j \right)du_{j_d} \nonumber \\
&& =2^{|A\cup B|} (-1)^{|B_{>d}|}\left(\frac{-1}{2}\right)\int_{\substack{(\mathbb{R}\geq0)^{|B|}\\\sum_{i\in A}u_i+\sum_{j\in B_{<d}} u_j \leq(\sum_{j \in B_{>d}} u_j)+u_{j_d}-1 } } \left( \prod_{i\in A\cup B} \hat{f_i}(u_i)\right)  \nonumber \\
&&\qquad \qquad \qquad \qquad\times \left(\prod_{j\in A\cup B} du_j \right)\nonumber \\
&& +2^{|A\cup B|} (-1)^{|B_{>d}|}\left(\frac{-1}{2}\right)  \int_{\substack{(\mathbb{R}\geq0)^{|B|}\\(\sum_{i\in A}u_i+\sum_{j\in B_{<d}} u_j)+u_{j_d} \leq(\sum_{j \in B_{>d}} u_j)-1 } } \left( \prod_{i\in A\cup B} \hat{f_i}(u_i)\right)  \nonumber \\
&&\qquad \qquad \qquad \qquad\times \left(\prod_{j\in A\cup B} du_j \right).
\end{eqnarray}
For the final equality we have split the $u_{j_d}$ integral into the intervals $[0,\infty)$ and $(-\infty,0)$ and made a change of variables $u_{j_d}\rightarrow -u_{j_d}$ in the second integral.  We see that the integral corresponding to $[0,\infty)$ for $d$ cancels with the integral corresponding to $(-\infty,0)$ in the case $d-1$. Similarly, the integral corresponding to $(-\infty,0)$ for $d$ cancels with the $[0,\infty)$ term for $d+1$.  In the case that $j_d$ is the largest element of $B$, the term corresponding to $(-\infty,0)$ is zero  because there is no way to satisfy the condition $(\sum_{i\in A}u_i+\sum_{j\in B_{<d}} u_j)+u_{j_d} \leq(\sum_{j \in B_{>d}} u_j)-1$ because all the $u$s are positive and $B_{>d}$ is the empty set. 

Thus the only term that survives is 
\begin{equation}
2^{|A\cup B|} (-1)^{|B|}\left(\frac{-1}{2}\right)  \int_{\substack{(\mathbb{R}\geq 0)^{|B|}\\\sum_{i\in A} u_i\leq (\sum_{j\in B} u_j)-1} }\left( \prod_{i\in A\cup B} \hat{f_i}(u_i)\right)   \left(\prod_{j\in A\cup B} du_j \right)
\end{equation} from (\ref{eq:lem4d}) and this proves the lemma.
\end{proof}

The final lemma is only needed when support is extended to $\sum_{i=1}^n |u_i|<3$.

\begin{proof}[Proof of Lemma \ref{lem:5}]
The proof of this lemma goes through exactly the same steps as the proof of Lemma \ref{lem:4}.  Taking the large $N$ limit, spreading the contours of integration so that $0<\delta_1<\cdots<\delta_m$ (noting as before that there are no poles at $z_i=z_j$ when the whole double sum is considered), and then replacing the $f$'s with their expression in terms of their Fourier transforms we find the left hand side of the statement of the lemma is
\begin{eqnarray}\label{eq:lem5a}
&&\sum_{d\in B_1}\sum_{g \in B_2}\frac{4}{(2\pi i)^2}\int_{(\delta_d)}\int_{(\delta_g)} \left(-e^{-2\pi(z_d+z_g)}\frac{(z_d-z_g) (z_g-z_d)} {(z_d+z_g) (z_d+z_g)(2z_d)(2z_g)}\right) \nonumber\\
&&\qquad\times \left( \prod_{i\in A_1} \frac{2}{2\pi i}\int_{(\delta_i)}\frac{1}{z_i+z_d}f_i( iz_i) dz_i\right)\times\left(\prod_{j\in B_1\backslash\{d\}}\frac{2}{2\pi i} \int_{(\delta_j)} \frac{-1}{z_j-z_d}f_j(iz_j) dz_j \right)\nonumber \\
&&\qquad\times \left( \prod_{i\in A_2} \frac{2}{2\pi i}\int_{(\delta_i)}\frac{1}{z_i+z_g}f_i( iz_i) dz_i\right)\times\left(\prod_{j\in B_2\backslash\{g\}}\frac{2}{2\pi i} \int_{(\delta_j)} \frac{-1}{z_j-z_g}f_j(iz_j) dz_j \right)\nonumber \\
&&\qquad\qquad\qquad\qquad\times f_d(iz_d) f_g( iz_g)dz_g dz_d\nonumber \\
&&=\sum_{d\in B_1}\sum_{g \in B_2}\int_{-\infty}^{\infty}\cdots \int_{-\infty}^{\infty} \prod_{i\in \{1,\ldots,m\}} \hat{f_i}(u_i) \;\frac{4}{(2\pi i)^2}\int_{(\delta_d)}\int_{(\delta_g)} -e^{2\pi(z_d(u_d-1)+z_g(u_g-1))}\nonumber \\
&&\qquad\qquad\qquad\qquad\qquad\qquad\times\frac{(z_d-z_g) (z_g-z_d)} {(z_d+z_g) (z_d+z_g)(2z_d)(2z_g)}\nonumber\\
&&\qquad\times \left( \prod_{i\in A_1} \frac{2}{2\pi i}\int_{(\delta_i)}\frac{e^{2\pi z_iu_i}}{z_i+z_d}dz_i\right)\times\left(\prod_{j\in B_1\backslash\{d\}}\frac{2}{2\pi i} \int_{(\delta_j)} \frac{-e^{2\pi z_ju_j}}{z_j-z_d} dz_j \right)\nonumber \\
&&\qquad\times \left( \prod_{i\in A_2} \frac{2}{2\pi i}\int_{(\delta_i)}\frac{e^{2\pi z_iu_i}}{z_i+z_g} dz_i\right)\times\left(\prod_{j\in B_2\backslash\{g\}}\frac{2}{2\pi i} \int_{(\delta_j)} \frac{-e^{2\pi z_ju_j}}{z_j-z_g}dz_j \right)dz_g dz_d \;du_1\cdots du_m.
\end{eqnarray}

We now perform the contour integration for all variables except $z_d$ and $z_g$, paying careful attention to the ordering $0<\delta_1<\cdots<\delta_m$ as at (\ref{eq:lem4e}).  So (\ref{eq:lem5a}) equals
\begin{eqnarray}\label{eq:lem5b}
&&\frac{2^m}{(2\pi i)^2}\sum_{d\in B_1}\sum_{g \in B_2}\int_{-\infty}^{\infty}\int_{-\infty}^{\infty}\;\int_{(\mathbb{R}\geq 0)^{\{1,\ldots,m\}\backslash\{d,g\}} }\prod_{i\in \{1,\ldots,m\}} \hat{f_i}(u_i) \int_{(\delta_d)}\int_{(\delta_g)} -e^{2\pi(z_d(u_d-1)+z_g(u_g-1))}\nonumber \\
&&\qquad\qquad\qquad\qquad\qquad\qquad\times\frac{(z_d-z_g) (z_g-z_d)} {(z_d+z_g) (z_d+z_g)(2z_d)(2z_g)}\nonumber\\
&&\qquad\times \left( \prod_{i\in A_1}e^{-2\pi z_du_i}\right)\left(\prod_{j\in B^1_{<d}}e^{-2\pi z_du_j} \right)\left(\prod_{j\in B^1_{>d}}-e^{2\pi z_du_j} \right)\nonumber \\
&&\qquad\times \left( \prod_{i\in A_2}e^{-2\pi z_gu_i}\right)\times\left(\prod_{j\in B^2_{<g}}e^{-2\pi z_gu_j} \right)\left(\prod_{j\in B^2_{>g}}-e^{2\pi z_gu_j} \right)\nonumber \\
&&\qquad\times dz_g dz_d \left( \prod_{j\in \{1,\ldots,m\}\backslash\{d,g\}} du_j\right) du_d du_g,
\end{eqnarray}
where here $B_{<d}^1$ is the set $\{j\in B_1: j<d\}$, and similarly with the other notation. 

The variable $z_g$ has a pole at $z_g=0$ and another at $z_g=-z_d$.  Computing these residues, we find that
\begin{eqnarray}\label{eq:lem5res1}
&&\Res_{z_g=0} \left( -e^{2\pi z_g(u_g-1-\sum_{i\in A_2}u_i-\sum_{j\in B_{<g}^2}u_j+\sum_{j\in B_{>g}^2} u_j)}  \frac{(z_d-z_g) (z_g-z_d)} {(z_d+z_g) (z_d+z_g)(2z_d)(2z_g)}\right) \nonumber \\
&& = \frac{1}{4z_d} 
\end{eqnarray}
and, with some miraculous cancelation,
\begin{eqnarray}\label{eq:lem5res2}
&&\Res_{z_g=-z_d} \left( -e^{2\pi z_g(u_g-1-\sum_{i\in A_2}u_i-\sum_{j\in B_{<g}^2}u_j+\sum_{j\in B_{>g}^2} u_j)}  \frac{(z_d-z_g) (z_g-z_d)} {(z_d+z_g) (z_d+z_g)(2z_d)(2z_g)}\right) \nonumber \\
&&=-2\pi \left(u_g-1-\sum_{i\in A_2}u_i-\sum_{j\in B_{<g}^2}u_j+\sum_{j\in B_{>g}^2} u_j\right) e^{-2\pi z_d(u_g-1-\sum_{i\in A_2}u_i-\sum_{j\in B_{<g}^2}u_j+\sum_{j\in B_{>g}^2} u_j)} .
\end{eqnarray}

These poles both lie to the left of the $\delta_g$ line, so we only pick up a contribution when $0<u_g-1-\sum_{i\in A_2}u_i-\sum_{j\in B_{<g}^2}u_j+\sum_{j\in B_{>g}^2} u_j$, meaning we are able to close the contour to the left, making the $z_g$ integral equal to $2\pi i$ times the sum of the enclosed residues.  Thus we have that after performing the $z_g$ integral, (\ref{eq:lem5b}) equals 
\begin{eqnarray} \label{eq:lem5b1}
&&2^m\sum_{d\in B_1}\sum_{g \in B_2}\int_{-\infty}^{\infty}\int_{-\infty}^{\infty}\;\int_{\substack{(\mathbb{R}\geq 0)^{\{1,\ldots,m\}\backslash\{d,g\}} \\\sum_{i\in A_2}u_i+\sum_{j\in B_{<g}^2}u_j<u_g-1+\sum_{j\in B_{>g}^2} u_j}}\prod_{i\in \{1,\ldots,m\}}  \hat{f_i}(u_i) \nonumber\\
&&\times (-1)^{|B_{>d}^1|} (-1)^{|B_{>g}^2|} \frac{1}{2\pi i} \int_{\delta_d} e^{2\pi(z_d(u_d-1-\sum_{i\in A_1}u_i-\sum_{j\in B_{<d}^1}u_j+\sum_{j\in B_{>d}^1} u_j)} \nonumber \\
&&\qquad \times \left( \frac{1}{4z_d}- 2\pi \left(u_g-1-\sum_{i\in A_2}u_i-\sum_{j\in B_{<g}^2}u_j+\sum_{j\in B_{>g}^2} u_j\right)\right.\nonumber \\
&& \qquad \times\left. e^{-2\pi z_d(u_g-1-\sum_{i\in A_2}u_i-\sum_{j\in B_{<g}^2}u_j+\sum_{j\in B_{>g}^2} u_j)} \right) dz_d\left( \prod_{j\in \{1,\ldots,m\}\backslash\{d,g\}} du_j\right) du_d du_g.
\end{eqnarray}

We now address the $z_d$ integral and (\ref{eq:lem5b1}) equals:
\begin{eqnarray}\label{eq:lem5c}
&&2^m\sum_{d\in B_1}\sum_{g \in B_2}\int_{-\infty}^{\infty}\int_{-\infty}^{\infty}\;\int_{\substack{(\mathbb{R}\geq 0)^{\{1,\ldots,m\}\backslash\{d,g\}} \\\sum_{i\in A_1}u_i+\sum_{j\in B_{<g}^1}u_j<u_d-1+\sum_{j\in B_{>g}^1} u_j\\\sum_{i\in A_2}u_i+\sum_{j\in B_{<g}^2}u_j<u_g-1+\sum_{j\in B_{>g}^2} u_j}}\prod_{i\in \{1,\ldots,m\}} \hat{f_i}(u_i)\nonumber \\
&&\qquad\qquad \times\frac{1}{4}(-1)^{|B_{>d}^1|} (-1)^{|B_{>g}^2|}  \left( \prod_{j\in \{1,\ldots,m\}\backslash\{d,g\}} du_j\right) du_d du_g\nonumber \\
&&-2^m\sum_{d\in B_1}\sum_{g \in B_2}\int_{-\infty}^{\infty}\int_{-\infty}^{\infty}\;\int_{\substack{(\mathbb{R}\geq 0)^{\{1,\ldots,m\}\backslash\{d,g\}} \\\sum_{i\in A_2}u_i+\sum_{j\in B_{<g}^2}u_j<u_g-1+\sum_{j\in B_{>g}^2} u_j}}\prod_{i\in \{1,\ldots,m\}} \hat{f_i}(u_i)\nonumber \\
&&\qquad\qquad \times\left(\int_{-\infty}^\infty \left(u_g-1-\sum_{i\in A_2}u_i-\sum_{j\in B_{<g}^2}u_j+\sum_{j\in B_{>g}^2} u_j\right)\right. \nonumber \\
&&\qquad\qquad \times\left. e^{2\pi iz_d(u_d-u_g-\sum_{i\in A_1\cup B_{<d}^1\cup B_{>g}^2}u_i+\sum_{j\in B_{>d}^1\cup A_2\cup B_{<g}^2} u_j)}dz_d \right)\nonumber \\
&&\qquad\qquad\times(-1)^{|B_{>d}^1|} (-1)^{|B_{>g}^2|}  \left( \prod_{j\in \{1,\ldots,m\}\backslash\{d,g\}} du_j\right) du_d du_g.
\end{eqnarray}
This simplifies, using the definition of the delta function at (\ref{eq:delta}), to
\begin{eqnarray}
&&2^m\sum_{d\in B_1}\sum_{g \in B_2}\int_{-\infty}^{\infty}\int_{-\infty}^{\infty}\;\int_{\substack{(\mathbb{R}\geq 0)^{\{1,\ldots,m\}\backslash\{d,g\}} \\\sum_{i\in A_1}u_i+\sum_{j\in B_{<g}^1}u_j<u_d-1+\sum_{j\in B_{>g}^1} u_j\\\sum_{i\in A_2}u_i+\sum_{j\in B_{<g}^2}u_j<u_g-1+\sum_{j\in B_{>g}^2} u_j}}\prod_{i\in \{1,\ldots,m\}} \hat{f_i}(u_i)\nonumber \\
&&\qquad\qquad \times\left(\frac{1}{4}-\left[\left(u_g-1-\sum_{i\in A_2}u_i-\sum_{j\in B_{<g}^2}u_j+\sum_{j\in B_{>g}^2} u_j\right)\right.\right.\nonumber \\
&&\qquad\qquad \times\left. \delta(u_d-u_g-\sum_{i\in A_1\cup B_{<d}^1\cup B_{>g}^2}u_i+\sum_{j\in B_{>d}^1\cup A_2\cup B_{<g}^2} u_j)\right]\nonumber \\
&&\qquad\qquad\times(-1)^{|B_{>d}^1|} (-1)^{|B_{>g}^2|}  \left( \prod_{j\in \{1,\ldots,m\}\backslash\{d,g\}} du_j\right) du_d du_g.
\end{eqnarray}
We note that above we can impose the condition $\sum_{i\in A_1}u_i+\sum_{j\in B_{<g}^1}u_j<u_d-1+\sum_{j\in B_{>g}^1} u_j$ on both integrals from (\ref{eq:lem5c}) because the delta function forces it to hold if the condition $\sum_{i\in A_2}u_i+\sum_{j\in B_{<g}^2}u_j<u_g-1+\sum_{j\in B_{>g}^2} u_j$ holds. 

Similarly to the proof of Lemma \ref{lem:4}, changing the sign of the variable on the negative half of the line of integration of  $u_d$ and $u_g$ results in a change of sign of that variable in the inequalities governing the region of integration and in the factor in square brackets above, but not in the $\hat{f}$s as the Fourier transforms of the test functions are each even functions.  We will denote the quantity in the square brackets above as $SQ(u_d,u_g)$ as a space saving device.  $SQ$ depends on the other $u$ variables as well, but $u_d$ and $u_g$ are the only ones that will change sign in the next equation so we list only those two explicitly.  Thus we end up with four terms
\begin{eqnarray}\label{eq:lem5d}
&&2^m\sum_{d\in B_1}\sum_{g \in B_2}\int_{0}^{\infty}\int_{0}^{\infty}\;\int_{\substack{(\mathbb{R}\geq 0)^{\{1,\ldots,m\}\backslash\{d,g\}} \\\sum_{i\in A_1}u_i+\sum_{j\in B_{<g}^1}u_j<u_d-1+\sum_{j\in B_{>g}^1} u_j\\\sum_{i\in A_2}u_i+\sum_{j\in B_{<g}^2}u_j<u_g-1+\sum_{j\in B_{>g}^2} u_j}}\prod_{i\in \{1,\ldots,m\}} \hat{f_i}(u_i)\nonumber \\
&&\qquad\qquad \times\left(\frac{1}{4}-SQ(u_d,u_g)\right)(-1)^{|B_{>d}^1|} (-1)^{|B_{>g}^2|}  \left( \prod_{j\in \{1,\ldots,m\}\backslash\{d,g\}} du_j\right) du_d du_g
\nonumber \\
&&+2^m\sum_{d\in B_1}\sum_{g \in B_2}\int_{0}^{\infty}\int_{0}^{\infty}\;\int_{\substack{(\mathbb{R}\geq 0)^{\{1,\ldots,m\}\backslash\{d,g\}} \\u_d+\sum_{i\in A_1}u_i+\sum_{j\in B_{<g}^1}u_j<-1+\sum_{j\in B_{>g}^1} u_j\\\sum_{i\in A_2}u_i+\sum_{j\in B_{<g}^2}u_j<u_g-1+\sum_{j\in B_{>g}^2} u_j}}\prod_{i\in \{1,\ldots,m\}} \hat{f_i}(u_i)\nonumber \\
&&\qquad\qquad \times\left(\frac{1}{4}-SQ(-u_d,u_g)\right)(-1)^{|B_{>d}^1|} (-1)^{|B_{>g}^2|}  \left( \prod_{j\in \{1,\ldots,m\}\backslash\{d,g\}} du_j\right) du_d du_g
\nonumber \\
&&+2^m\sum_{d\in B_1}\sum_{g \in B_2}\int_{0}^{\infty}\int_{0}^{\infty}\;\int_{\substack{(\mathbb{R}\geq 0)^{\{1,\ldots,m\}\backslash\{d,g\}} \\\sum_{i\in A_1}u_i+\sum_{j\in B_{<g}^1}u_j<u_d-1+\sum_{j\in B_{>g}^1} u_j\\u_g+\sum_{i\in A_2}u_i+\sum_{j\in B_{<g}^2}u_j<-1+\sum_{j\in B_{>g}^2} u_j}}\prod_{i\in \{1,\ldots,m\}} \hat{f_i}(u_i)\nonumber \\
&&\qquad\qquad \times\left(\frac{1}{4}-SQ(u_d,-u_g)\right)(-1)^{|B_{>d}^1|} (-1)^{|B_{>g}^2|}  \left( \prod_{j\in \{1,\ldots,m\}\backslash\{d,g\}} du_j\right) du_d du_g
\nonumber \\
&&+2^m\sum_{d\in B_1}\sum_{g \in B_2}\int_{0}^{\infty}\int_{0}^{\infty}\;\int_{\substack{(\mathbb{R}\geq 0)^{\{1,\ldots,m\}\backslash\{d,g\}} \\u_d+\sum_{i\in A_1}u_i+\sum_{j\in B_{<g}^1}u_j<-1+\sum_{j\in B_{>g}^1} u_j\\u_g+\sum_{i\in A_2}u_i+\sum_{j\in B_{<g}^2}u_j<-1+\sum_{j\in B_{>g}^2} u_j}}\prod_{i\in \{1,\ldots,m\}} \hat{f_i}(u_i)\nonumber \\
&&\qquad\qquad \times\left(\frac{1}{4}-SQ(-u_d,-u_g)\right)(-1)^{|B_{>d}^1|} (-1)^{|B_{>g}^2|}  \left( \prod_{j\in \{1,\ldots,m\}\backslash\{d,g\}} du_j\right) du_d du_g.
\end{eqnarray}
The right hand side of Lemma \ref{lem:5} is the first term in (\ref{eq:lem5d}) in the case that $d$ is the smallest element in $B_1$ and $g$ is the smallest element in $B_2$. 

To see that all the other terms cancel out we start by labelling the elements of $B_1$ as $j_1<j_2<\cdots<j_{|B_1|}$ and the elements of $B_2$ as $k_1<k_2<\cdots<k_{|B_2|}$. Then we pick a single term in the double sum with conditions $\sum _{i\in A_1} u_i\; +u_{j_1}+\cdots +u_{j_p} < u_{j_{p+1}}+\cdots+u_{|B_1|}-1$ and $\sum _{i\in A_1} u_i \;+u_{k_1}+\cdots +u_{k_q} < u_{k_{q+1}}+\cdots+u_{|B_2|}-1$.  That term can arise in four different ways.  It could come from the term in the double sum where $d=j_{p+1}$ and $g=k_{q+1}$.  We will refer to this term as carrying an overall plus sign.  Also, the same term will appear when $d=j_p$ and $g=k_{q+1}$ (this would be a term like that in the second line of (\ref{eq:lem5c})). It would then have an overall minus sign because there would be more element in $B_{>d}^1$. The same expression would come up when $d=j_{p+1}$ and $g=k_{q}$ (as in the third line above): again this would occur with an extra minus sign.  The four occurance would be in the term $d=j_p$ and $g=k_q$.  This would have a plus sign.  Thus we can see that these four terms would cancel.  The only situation under which these four terms would not exist is if $p+1=1$ or if $q+1=1$.  If $p+1=1$ but $q+1\neq 1$ then there exists no term with $d=p$ and only two of the four terms exist, but those two still cancel as their signs are opposite.  Similarly if $p+1\neq1$ and $q+1= 1$.  If both $p+1=1$ and $q+1=1$ then this is the term that gives us the right hand side of Lemma \ref{lem:5}.  (It is also helpful to note that there is no term where $j_p$ or $k_q$ are the largest elements in the sets $B_1$ or $B_2$, respectively, because in this case the inequality can never be satisfied.)

Thus we are left with just the first term in (\ref{eq:lem5c}) in the case that $d$ is the smallest element in $B_1$ and $g$ is the smallest element in $B_2$:
\begin{eqnarray}
&&2^{m}(-1)^{|B_1\cup B_2|} \int_{\substack{(\mathbb{R}\geq 0)^{|A_1\cup B_1\cup A_2 \cup B_2|} \\ \sum_{i\in A_1}u_i\leq (\sum_{i\in B_1} u_i)-1\\ \sum_{i\in A_2}u_i\leq (\sum_{i\in B_2} u_i)-1}} \left( \frac{1}{4}-\left[\left(-\sum_{i\in A_2}u_i+\sum_{j\in B_2} u_j-1\right)\right.\right.\nonumber \\
&&\qquad\qquad \times\left.\left. \delta\left(-\sum_{i\in A_1\cup B_2}u_i+\sum_{j\in B_1\cup A_2} u_j\right)\right]\right)\prod_{i\in \{1,\ldots,m\}} \hat{f_i}(u_i) \prod_{i\in \{1,\ldots,m\}} du_i.
\end{eqnarray}
\end{proof}

\section{ Support $\sum |u_j| <1$} \label{sect:support1}

In this Section we demonstrate that with the Lemmas \ref{lem:1} to \ref{lem:3} it is straightforward to show that Rubinstein's $n$-level density \cite{kn:rub01} of the zeros of quadratic Dirichlet $L$-functions agrees in the limit with the random matrix result for the $n$-level density of eigenvalues from large matrices selected with Haar measure from $USp(2N)$.  Rubinstein's result states the $n$-level density for test functions with Fourier transform, $\prod_{i=1}^n \hat{f_i}(u_i)$, supported on $\sum_{i=1}^n |u_i|<1$.  He proves that the limit matches random matrix theory, but here we show that the job is made easier by comparing with (\ref{eq:Uworking2}) instead of the determinantal form of the $n$-level density. 

That is, we give a straightforward proof of the following:

\begin{theorem}\label{theo:support1}
For test functions $f_1, \ldots, f_n$ with properties as at (\ref{eq:testfunction}) and with the product of their Fourier transforms $\prod_{i=1}^n \hat{f_i}(u_i)$ having support in $\sum_{i=1}^n |u_i|<1$, 
\begin{eqnarray}
&&\lim_{X\rightarrow \infty} \frac{1}{|D(X)|}  \sum_{d\in D(X)} \sum_{j_1,\ldots,j_n=-\infty}^\infty f_1\Big(L\gamma_d^{(j_1)} \Big)\cdots f_n\Big(L\gamma_d^{(j_n)}\Big)\nonumber \\
&& = \lim_{N\rightarrow \infty}  \int_{USp(2N)} \sum _{\substack{j_1,\ldots,j_n=-\infty \\j_1,\ldots,j_n\neq 0}}^\infty f_1\Big(\frac{N}{\pi} \theta_{j_1}\Big) \cdots f_n\Big(\frac{N}{\pi} \theta_{j_n}\Big) d\mu_{{\rm Haar}},
\end{eqnarray}
where the eigenangles $\theta_{j_i}$ are counted as described at Theorem \ref{theo:massna}.
\end{theorem}

In \cite{kn:rub01}, Rubinstein's first step is to prepare a sum over zeros that are distinct.  We do not need to make this combinatorial argument because, unlike the determinantal form of the $n$-level density,  the result in (\ref{eq:Uworking2})  sums over all $n$-tuples of zeros, not just those containing distinct zeros.  Thus we rewrite Rubinstein's result to make this simplification. 

Rubinstein chooses a test function of the form
\begin{equation}
f(x_1,x_2, \ldots, x_n)=\prod_{i=1}^n f_i(x_i),
\end{equation}
where each $f_i$ is even and in $S(\mathbb{R})$ (i.e., smooth and rapidly decreasing).  Removing the restriction of summing over distinct zeros, the quantity Rubinstein calculates is the $n$-level density
\begin{theorem}[Rubinstein \cite{kn:rub01}]\label{theo:rubinstein} With the product of the Fourier transforms of the test functions, $\prod_{i=1}^n \hat{f_i}(u_i)$, having support in $\sum_{i=1}^n |u_i|<1$, 
\begin{eqnarray}\label{eq:rubinstein}
&&\lim_{X\rightarrow \infty} \frac{1}{|D(X)|}  \sum_{d\in D(X)} \sum_{j_1,\ldots,j_n=-\infty}^\infty f_1\Big(L\gamma_d^{(j_1)} \Big)\cdots f_n\Big(L\gamma_d^{(j_n)}\Big)\nonumber \\
&&=\sum_{Q\cup M=\{1,\ldots,n\}} \Bigg(\prod_{m\in M}\int_{-\infty}^\infty f_m(x)dx\Bigg) \nonumber \\
&&\qquad\qquad \times\Bigg( \sum_{\substack{S_2\subseteq Q\\|S_2|\;{\rm even}}} \bigg(\Big(-\frac{1}{2}\Big)^{|S_2^c|} \prod_{\ell \in  S_2^c} \int_{-\infty}^\infty \hat{f_\ell}(u)du\bigg) \bigg( \sum_{(A;B)} 2^{|S_2|/2} \prod_{j=1}^{|S_2|/2} \int_{-\infty}^\infty |u| \hat{f_{a_j}}(u) \hat{f_{b_j}}(u) du \bigg) \Bigg).
\end{eqnarray}
where notation is given at the start of Section \ref{sect:results} and some explanation is given below.
\end{theorem}

Rubinstein uses the explicit formula in the form
\begin{equation}
\sum_{j=-\infty}^\infty f(L\gamma_d^{j}) = \int_{-\infty}^\infty f(x)dx +O\Big(\frac{1}{\log X}
\Big) -\frac{2}{\log X} \sum _{m=1}^\infty \frac{\Lambda(m)}{m^{1/2}} \chi_d(m) \hat{f}\Big(\frac {\log m} {\log X} \Big)
\end{equation} 
to perform the following manipulations
\begin{eqnarray}
&&\lim_{X\rightarrow \infty} \frac{1}{|D(X)|}  \sum_{d\in D(X)} \sum_{j_1,\ldots,j_n=-\infty}^\infty f_1\Big(L\gamma_d^{(j_1)} \Big)\cdots f_n\Big(L\gamma_d^{(j_n)}\Big)\nonumber \\
&&= \lim_{X\rightarrow \infty} \frac{1}{|D(X)|} \prod_{j=1}^n \Bigg(\int_{-\infty}^\infty f_j(x)dx +O\Big(\frac{1}{\log X}
\Big) -\frac{2}{\log X} \sum _{m=1}^\infty \frac{\Lambda(m)}{m^{1/2}} \chi_d(m) \hat{f_j}\Big(\frac {\log m} {\log X} \Big)
\Bigg)\nonumber \\
&& = \sum_{Q\cup M=\{1,\ldots,n\}} \Bigg(\prod_{m\in M}\int_{-\infty}^\infty f_m(x)dx\Bigg) \nonumber \\
&&\qquad\qquad \times \Bigg( \lim_{X\rightarrow \infty} \frac{1}{|D(X)|} \sum_{d\in D(X)} \prod_{k\in Q} \bigg( -\frac{2}{\log X} \sum _{m=1}^\infty \frac{\Lambda(m)}{m^{1/2}} \chi_d(m) \hat{f_k}\Big(\frac {\log m} {\log X} \bigg)\Bigg)\nonumber \\
&&=\sum_{Q\cup M=\{1,\ldots,n\}} \Bigg(\prod_{m\in M}\int_{-\infty}^\infty f_m(x)dx\Bigg) \nonumber \\
&&\qquad\qquad \times\Bigg( \sum_{\substack{S_2\subseteq Q\\|S_2|\;{\rm even}}} \bigg(\Big(-\frac{1}{2}\Big)^{|S_2^c|} \prod_{\ell \in  S_2^c} \int_{-\infty}^\infty \hat{f_\ell}(u)du\bigg) \bigg( \sum_{(A;B)} 2^{|S_2|/2} \prod_{j=1}^{|S_2|/2} \int_{-\infty}^\infty |u| \hat{f_{a_j}}(u) \hat{f_{b_j}}(u) du \bigg) \Bigg).
\end{eqnarray}
The last line follows from Rubinstein's Lemma 1 \cite{kn:rub01} and we will not go into details here.  The sum over $Q$ and $M$ is over disjoint subsets of $\{1,\ldots,n\}$.  The sum over $S_2$ is over all subsets of $Q$ whose size is even. $S_2^c$ is the complement of $S_2$ in $Q$.  $\sum_{(A;B)}$ is a sum over all the ways of a pairing up the elements of $S_2$.  To use Rubinstein's example, if $Q=\{1,2,5,7\}$, the possible $S_2$'s are $\emptyset, \{1,2\}, \{1,5\},\{1,7\}, \{2,5\}, $ $\{2,7\}, \{5,7\}, \{1,2,5,7\}$.  If $S_2=\{1,2,5,7\}$ then the possible $(A;B)$'s are $(1,2;5,7)$, $(1,2;7,5)$, and $(1,5;2,7)$.  These correspond, respectively, to matching 1 with 5, 2 with 7, 1 with 7 and 2 with 5, 1 with 2 and 5 with 7.  Note that this notation is not unique, but is sufficient for our purposes. 

\begin{proof}[Proof of Theorem \ref{theo:support1}]
We start  with the expression for the $n$-level density of eigenvalues of matrices from $USp(2N)$ given at (\ref{eq:Uworking2})
\begin{eqnarray}\label{eq:Js}
&& \lim_{N\rightarrow \infty}  \int_{USp(2N)} \sum _{\substack{j_1,\ldots,j_n=-\infty \\j_1,\ldots,j_n\neq 0}}^\infty f_1\Big(\frac{N}{\pi} \theta_{j_1}\Big) \cdots f_n\Big(\frac{N}{\pi} \theta_{j_n}\Big) d\mu_{{\rm Haar}}  \nonumber\\
&&= \lim_{N\rightarrow \infty} \frac{1}{(2\pi i)^n} \\
&&\qquad\times \sum_{Q\cup M = \{1,\ldots,n\}} (2N)^{|M|}\int_{(\delta)^{|Q|}} \int_{(0)^{|M|}} 2^{|Q|} J_{USp(2N)}^*(z_Q) f_1\Big(\frac{N}{\pi}iz_1\Big) \cdots f_n\Big(\frac{N}{\pi}iz_n\Big) dz_1\cdots dz_n.\nonumber
\end{eqnarray}
Since the support of the Fourier transforms is in $\sum_{i=1}^n |u_i|<1$, we can replace $J_{USp(2N)}^*(z_Q)$ with $J_{USp(2N),1}^*(z_Q)$ as described at (\ref{def:UJq}).  We also separate out of the integral the variables in $M$ and use Lemma \ref{lem:3} to tidy them up.  This leads us to
\begin{eqnarray}\label{eq:QM}
&&\sum_{Q\cup M = \{1,\ldots,n\}} \Bigg( \prod_{m\in M} \int_{-\infty}^\infty f_m(x)dx\Bigg)\nonumber \\ &&\qquad\qquad \times \Bigg(\lim_{N\rightarrow \infty} \frac{2^{|Q|}}{(2\pi i)^{|Q|}} \int_{(\delta)^{|Q|}} J_{USp(2N),1}^* (z_Q) \prod_{k\in Q} f_k\Big( \frac{N}{\pi}iz_k\Big) dz_Q \Bigg) \\
&&=\sum_{Q\cup M = \{1,\ldots,n\}} \Bigg( \prod_{m\in M} \int_{-\infty}^\infty f_m(x)dx\Bigg) \nonumber \\
&&\qquad\qquad \times\Bigg(\lim_{N\rightarrow \infty} \frac{1}{(2\pi i)^{|Q|}} \int_{(\delta)^Q} \sum_{\substack{Q=W_1\cup\cdots\cup W_r\\ |W_r|\leq 2}} 2^{|Q|} \prod_{r=1}^R H_{\emptyset}(W_r)\; \prod_{k\in Q} f_k\Big(\frac{N}{\pi} iz_k\Big) dz_Q\Bigg) ,
\end{eqnarray}
using the definition of $J_{USp(2N)}^*$ from Theorem \ref{theo:massna}. Note that the sum over the $W_r$ splits $Q$ up into pairs and singletons of $z$s, just as in Rubinstein's result in Theorem \ref{theo:rubinstein}.

Now we rewrite this using Rubinstein's summing notation:
\begin{eqnarray}
&&\sum_{Q\cup M = \{1,\ldots,n\}} \Bigg( \prod_{m\in M} \int_{-\infty}^\infty f_m(x)dx\Bigg)  \Bigg(\sum_{\substack{S_2\subseteq Q\\ |S_2| \;{\rm even}}} \bigg(\prod_{\ell \in  S_2^c} \lim_{N\rightarrow \infty} \frac{2}{2\pi i} \int_{(\delta)} H_{\emptyset}(z_\ell) f_\ell\Big(\frac{N}{\pi} iz_\ell) dz_\ell\bigg) \nonumber \\
&& \qquad \qquad \times \bigg(\sum_{(A;B)} \prod_{j=1}^{|S_2|/2} \lim_{N\rightarrow \infty} \Big( \frac{4} {(2\pi i)^2} \int_{(\delta)} \int_{(\delta)} H_{\emptyset}(z_{a_j},z_{b_j}) f_{a_j}\Big(\frac{N}{\pi} iz_{a_j}\Big) f_{b_j}\Big(\frac{N}{\pi}iz_{b_j}\Big) dz_{a_j}dz_{b_j} \bigg)  \Bigg).
\end{eqnarray}

Note that
\begin{eqnarray}
H_\emptyset (z_\ell)&=& \frac{z'}{z} (2z_\ell)\\
H_{\emptyset}(z_1,z_2) &=& \Big(\frac{z'}{z}\Big)'(z_1+z_2),
\end{eqnarray}
and so we have, using Lemmas \ref{lem:1} and \ref{lem:2},
\begin{eqnarray}
&&\sum_{Q\cup M = \{1,\ldots,n\}} \Bigg( \prod_{m\in M} \int_{-\infty}^\infty f_m(x)dx\Bigg)  \Bigg(\sum_{\substack{S_2\subseteq Q\\ |S_2| \;{\rm even}}} \bigg(\prod_{\ell \in  S_2^c} -\frac{1}{2} \int_{-\infty}^\infty \hat{f_\ell}(u)du \bigg)\nonumber \\
&& \qquad \qquad\times  \bigg(\sum_{(A;B)} \prod_{j=1}^{|S_2|/2} 2\int_{-\infty}^\infty |u|  \hat{f_{a_j}}(u) \hat{f_{b_j}}(u)du \bigg)  \Bigg).
\end{eqnarray}
This is exactly the form of Rubinstein's Theorem \ref{theo:rubinstein}. 

\end{proof}
\section{Support $\sum |u_j|<2$}\label{sect:support2}

In this section we show that we can use (\ref{eq:Uworking2}) to match the $n$-level density of zeros from quadratic Dirichlet $L$-functions derived by Gao \cite{kn:gao14} to the $n$-level density of eigenvalues of matrices from $USp(2N)$.

We start by writing out Gao's result (Theorem II.1 of \cite{kn:gao05}, with minor changes of notation) without the restriction that the zeros must be distinct.  This can also be compared with Theorem 7.2 in the paper \cite{kn:entrodrud}, which is identical to Gao's expression. Both \cite{kn:gao14} and \cite{kn:entrodrud} use the family of Dirichlet $L$-functions $L(s,\chi_{8d})$ (with odd, positive, square-free $d$) as it simplifies the workings somewhat, but this does not make a material difference to the end result and does not affect the comparison with random matrix theory as the family is still expected to have zeros showing the behaviour of eigenvalues from $USp(2N)$. 

\begin{theorem}(Gao \cite{kn:gao14})\label{theo:gao}
Assume GRH and assume that $f_i$ is even and in $S(\mathbb{R})$ and $\prod_i \hat{f_i}(u_i)$ is supported in $\sum_{i=1}^n |u_i|<2$.  Then
\begin{eqnarray}\label{eq:gaotheorem}
&&\lim_{X\rightarrow \infty} \frac{\pi^2} {4X} \sum_{d\in D(X)} \sum_{j_1,\ldots,j_n} f_1(L\gamma_{8d}^{(j_1)} ) \cdots f_n(L\gamma_{8d}^{(j_n)}) \nonumber \\
&& = \sum_{Q\cup M=\{1,\ldots,n\}} \left( \prod_{m\in M} \int_{-\infty}^{\infty} f_m(x) dx\right) \left[ \sum_{S_2\subseteq Q} \left( \Big(\frac{-1}{2}\Big)^{|S_2^c|} \prod_{\ell \in S_2^c} \int_{-\infty}^{\infty} \hat{f_\ell}(u)du\right)\right. \nonumber \\
&&\times \left ( \Bigg( \frac{1+(-1)^{|S_2|}}{2}\Bigg) 2^{|S_2|/2} \sum_{S_2=(A:B)} \prod_{i=1}^{|S_2|/2} \int_{-\infty}^\infty |u_i| \hat{f_{a_i}} (u_i) \hat{f_{b_i}}(u_i) du_i \right.\nonumber \\
&&-\frac{1}{2} \sum_{\substack{S_3\subsetneq S_2\\ |S_3| {\rm \; even}}} 2^{|S_3|/2} \left( \sum_{S_3=(C:D)} \prod_{i=1}^{|S_3|/2} \int_{-\infty}^\infty |u_i| \hat{f_{c_i}}(u_i)\hat{f_{d_i}}(u_i) du_i\right)\nonumber \\
&&\left. \left.\times \sum_{I\subsetneq S_3^c}(-1)^{|I|} (-2)^{|S_3^c|}\int_{\substack{(\mathbb{R}\geq 0)^{S_3^c} \\ \sum_{i\in I}u_i\leq (\sum_{i\in I^c} u_i)-1}} \prod_{i\in S_3^c} \hat{f_i}(u_i) \prod_{i\in S_3^c} du_i\right) \right].
\end{eqnarray}
Here $Q\cup M=\{1,\ldots,n\}$ denotes partitioning $\{1,\ldots,n\}$ into two disjoint sets.  $S_2^c$ denotes the complement of $S_2$ in $Q$. The sum denoted $(A:B)$ indicates summing over all partitions $\{\{a_1,b_1\},\ldots\{a_{|S_2|/2},b_{|S_2|/2}\}\}$ of $S_2$.
\end{theorem}
In the above theorem we use the convention that empty products are 1 and empty sums are 0, but note that the empty set is an allowed set in a sum over sets.  In particular, if $S_3$ is empty then the sum over $(C:D)$ is 1.  If $Q$ is empty then the sum over $S_2$ is 1. If $S_2$ is empty then the sum over $S_3$ is zero. 

We aim to use (\ref{eq:Uworking2}) to prove the theorem:

\begin{theorem}\label{theo:support2}
For test functions $f_1, \ldots, f_n$ with the properties at (\ref{eq:testfunction}) and with the product of their Fourier transforms $\prod_{i=1}^n \hat{f_i}(u_i)$ having support in $\sum_{i=1}^n |u_i|<2$, 
\begin{eqnarray}
&&\lim_{X\rightarrow \infty} \frac{\pi^2} {4X} \sum_{d\in D(X)} \sum_{j_1,\ldots,j_n} f_1(L\gamma_{8d}^{(j_1)} ) \cdots f_n(L\gamma_{8d}^{(j_n)})\nonumber \\
&& = \lim_{N\rightarrow \infty}  \int_{USp(2N)} \sum _{\substack{j_1,\ldots,j_n=-\infty \\j_1,\ldots,j_n\neq 0}}^\infty f_1\Big(\frac{N}{\pi} \theta_{j_1}\Big) \cdots f_n\Big(\frac{N}{\pi} \theta_{j_n}\Big) d\mu_{{\rm Haar}},
\end{eqnarray}
where the eigenangles $\theta_{j_i}$ are counted as described at Theorem \ref{theo:massna}.
\end{theorem}

\begin{proof}[Proof of Theorem \ref{theo:support2}]
We start, as in the previous section, with (\ref{eq:Js}).  As the support of the Fourier transforms is in $\sum_{i=1}^n |u_i|<2$, this time we replace $J_{USp(2N)}^*(z_Q)$ with $J_{USp(2N),2}^*(z_Q)$ as described at (\ref{def:UJq}).  Thus, in analogy with (\ref{eq:QM}),  we have
\begin{eqnarray}\label{eq:J2z}
&&\lim_{N\rightarrow \infty}  \int_{USp(2N)} \sum _{\substack{j_1,\ldots,j_n=-\infty \\j_1,\ldots,j_n\neq 0}}^\infty f_1\Big(\frac{N}{\pi} \theta_{j_1}\Big) \cdots f_n\Big(\frac{N}{\pi} \theta_{j_n}\Big) d\mu_{{\rm Haar}}  \nonumber\\
&&=\sum_{Q\cup M = \{1,\ldots,n\}} \Bigg( \prod_{m\in M} \int_{-\infty}^\infty f_m(x)dx\Bigg)\nonumber \\ &&\qquad\qquad \times \Bigg(\lim_{N\rightarrow \infty} \frac{2^{|Q|}}{(2\pi i)^{|Q|}} \int_{(\delta)^{|Q|}} J_{USp(2N),2}^* (z_Q) \prod_{k\in Q} f_k\Big( \frac{N}{\pi}iz_k\Big) dz_Q \Bigg).
\end{eqnarray}

We will concentrate on the factor containing $J_{USp(2N),2}^* (z_Q) $. We will use (\ref{def:UJstar}) (or alternatively see (\ref{def:UJq})) and keep only terms with $|D|=0$ (first line of (\ref{eq:J2}) below) or 1 (second line below).  
\begin{eqnarray}\label{eq:J2}
J_{USp(2N),2}^* (z_Q)& =&  \sum_{\substack{R\subset Q \\ |R| {\rm \;even}}} \left( \prod_{\ell\in R^c} H_{\emptyset}(z_\ell)\right) \left(\sum_{R=(A:B)} \prod_{j=1}^{|R|/2} H_\emptyset (z_{a_j},z_{b_j})\right) \nonumber \\
 &&+\sum_{d\in Q}-e^{-2Nz_d}z(-2z_d) \sum_{\substack{R\subset Q\backslash\{d\}\\|R| {\rm \;even}}}  \left( \prod_{\ell\in R^c} H_{\{d\}}(z_\ell)\right) \left(\sum_{R=(A:B)} \prod_{j=1}^{|R|/2} H_{\{d\}}(z_{a_j},z_{b_j})\right) .
\end{eqnarray}

We now note that in the definition of 
\begin{equation}\label{eq:H_dzell}
H_{\{d\}}(z_\ell)= \frac{z'}{z}(z_\ell-z_d) -\frac{z'}{z}(z_\ell+z_d) +\frac{z'}{z}(2z_\ell)
\end{equation}
 there is the term $\frac{z'}{z}(2z_\ell)=:H_{\emptyset}(z_\ell)$.  Also it is clear that $H_{\{d\}}(z_{a_j},z_{b_j})=H_{\emptyset}(z_{a_j},z_{b_j})$.  In (\ref{eq:J2}) we multiply out the product over $H_{\{d\}}(z_\ell)$ and rearrange the terms, grouping together the $H_{\emptyset}(z_\ell)$ factors.  Again, in (\ref{eq:J2a}), the first line corresponds to terms with $|D|=0$ and the later lines to $|D|=1$. 
 \begin{eqnarray}\label{eq:J2a}
J_{USp(2N),2}^* (z_Q)& =& \sum_{S_2\subset Q} \left( \prod_{\ell \in S_2^c} \frac{z'}{z}(2z_\ell)\right) \left[\frac{1+(-1)^{|S_2|}}{2} \sum_{S_2 = (A:B)} \prod_{j=1}^{|S_2|/2} \left(\frac{z'}{z}\right)'(z_{a_j}+z_{b_j}) \right.\nonumber \\
&&+\sum_{\substack{S_3\subsetneq S_2\\|S_3|{\rm \;even}}} \left(\sum_{S_3=(C:D)} \prod_{j=1}^{|S_3|/2} \left(\frac{z'}{z}\right)'(z_{c_j}+z_{d_j})\right)\nonumber \\
&&\qquad\left. \times \left(\sum_{I\subsetneq S_3^c} \sum_{d\in I^c} \left( \prod_{i\in I} -\frac{z'}{z}(z_i+z_d)\right) \left(\prod_{j\in I^c\backslash\{d\}}\frac{z'}{z}(z_j-z_d) \right) \left(-e^{-2Nz_d}z(-2z_d)\right) \right)   \right].
\end{eqnarray}
We note that $S_3\neq S_2$, so there must be at least one element in $S_3^c$, the complement of $S_3$ in $S_2$. Similarly, there must be one element in $I^c$, so we know the sum over $d$ always exists, meaning we always have one variable playing the role of $d$ from the second line of (\ref{eq:J2}), while the other variables in $Q$ are spread between the factors of $H_{\emptyset}(z_\ell)$, $H_{\emptyset}(z_{a_j},z_{b_j})$ or factors of the form of the the first or second term of (\ref{eq:H_dzell}).  

Therefore (\ref{eq:J2z}) equals
\begin{eqnarray}
&&\sum_{Q\cup M = \{1,\ldots,n\}} \Bigg( \prod_{m\in M} \int_{-\infty}^\infty f_m(x)dx\Bigg)\nonumber \\ &&\qquad \times \Bigg(\lim_{N\rightarrow \infty}  \sum_{S_2\subset Q} \left( \prod_{\ell \in S_2^c} \frac{2}{2\pi i} \int_{(\delta)}\frac{z'}{z}(2z_\ell)f_\ell\Big( \tfrac{N}{\pi}iz_\ell\Big)dz_\ell\right) \nonumber \\
&& \qquad \times\left[\frac{1+(-1)^{|S_2|}}{2} \sum_{S_2 = (A:B)} \prod_{j=1}^{|S_2|/2}\frac{4}{(2\pi i)^2}\int_{(\delta)} \int_{(\delta)}  \left(\frac{z'}{z}\right)'(z_{a_j}+z_{b_j})f_{a_j}\Big( \tfrac{N}{\pi}iz_{a_j}\Big)f_{b_j}\Big( \tfrac{N}{\pi}iz_{b_j}\Big)dz_{a_j}dz_{b_j} \right.\nonumber \\
&&\qquad+\sum_{\substack{S_3\subsetneq S_2\\|S_3|{\rm \;even}}} \left(\sum_{S_3=(C:D)} \prod_{j=1}^{|S_3|/2}\frac{4}{(2\pi i)^2} \int_{(\delta)} 
\int_{(\delta)} \left(\frac{z'}{z}\right)'(z_{c_j}+z_{d_j})f_{c_j}\Big( \tfrac{N}{\pi}iz_{c_j}\Big)f_{d_j}\Big( \tfrac{N}{\pi}iz_{d_j}\Big)dz_{c_j}dz_{d_j}\right)\nonumber \\
&&\qquad \times \left(\sum_{I\subsetneq S_3^c} \sum_{d\in I^c}\frac{2}{2\pi i}\int_{(\delta)} \left(-e^{-2Nz_d}z(-2z_d)\right) \left( \prod_{i\in I} \frac{2}{2\pi i}\int_{(\delta)}-\frac{z'}{z}(z_i+z_d)f_i\Big( \tfrac{N}{\pi}iz_i\Big) dz_i\right)\right. \nonumber \\
&&\qquad\left.\left.\times\left(\prod_{j\in I^c\backslash\{d\}}\frac{2}{2\pi i} \int_{(\delta)} \frac{z'}{z}(z_j-z_d)f_j\Big( \tfrac{N}{\pi}iz_j\Big) dz_j \right) f_d\Big( \tfrac{N}{\pi}iz_d\Big) dz_d \right)   \right]  \Bigg)
\end{eqnarray}
We can now apply Lemmas \ref{lem:1}, \ref{lem:2} and \ref{lem:4} (with $A=I$ and $B=I^c$) to achieve the desired result. (Note that the set $S_3^c=I\cup I^c$ is not necessarily a set of consecutive integers $\{1, 2,\ldots,m\}$ as stated in Lemma \ref{lem:4}, but it it certainly a set of positive integers with no duplications, which one can easily imagine mapping onto $\{1, 2,\ldots,m\}$ without changing the proof of Lemma \ref{lem:4} in any way.)
\end{proof}

\section{Support $\sum|u_j|<3$} \label{sect:support3}

To our knowledge, there is no symplectic or orthogonal family of $L$-functions where the level densities can be written down when the support of the Fourier transform of the test function extends beyond $\sum|u_j|<2$ - including in the function field case.  So, in this section we write down the random matrix $n$-level densities that hold in the range $\sum|u_j|<3$ in a form similar to Gao's Theorem \ref{theo:gao}.  In doing this we hope it may lead to a number theoretical result being proven for a specific family.  We note that although more terms survive as the support increases, and so the formulae appear more unwieldy, in fact the ideas used in this section are just those used in the previous section for support in  $\sum|u_j|<2$.  So, it seems in the random matrix case there is no barrier, besides more convoluted formulae, to extending the support as far as need be. 

We start with the formula, analogous to (\ref{eq:QM}), featuring $J_{USp(2N),3}^* (z_Q)$ and valid for $\sum|u_j|<3$,
\begin{eqnarray}\label{eq:J3a}
&&\lim_{N\rightarrow \infty}  \int_{USp(2N)} \sum _{\substack{j_1,\ldots,j_n=-\infty \\j_1,\ldots,j_n\neq 0}}^\infty f_1\Big(\frac{N}{\pi} \theta_{j_1}\Big) \cdots f_n\Big(\frac{N}{\pi} \theta_{j_n}\Big) d\mu_{{\rm Haar}}  \nonumber\\
&&=\sum_{Q\cup M = \{1,\ldots,n\}} \Bigg( \prod_{m\in M} \int_{-\infty}^\infty f_m(x)dx\Bigg)\nonumber \\ &&\qquad\qquad \times \Bigg(\lim_{N\rightarrow \infty} \frac{2^{|Q|}}{(2\pi i)^{|Q|}} \int_{(\delta)^{|Q|}} J_{USp(2N),3}^* (z_Q) \prod_{k\in Q} f_k\Big( \frac{N}{\pi}iz_k\Big) dz_Q \Bigg).
\end{eqnarray}

Similarly to the previous section, we examine $J_{USp(2N),3}^* (z_Q) $. We use (\ref{def:UJstar}) and keep only terms with $|D|=0,1,2$ (see discussion at (\ref{def:UJq})):
\begin{eqnarray}\label{eq:J3}
J_{USp(2N),3}^* (z_Q)& =&  \sum_{\substack{R\subset Q \\ |R| {\rm \;even}}} \left( \prod_{\ell\in R^c} H_{\emptyset}(z_\ell)\right) \left(\sum_{R=(A:B)} \prod_{j=1}^{|R|/2} H_\emptyset (z_{a_j},z_{b_j})\right) \nonumber \\
 &&+\sum_{d\in Q}-e^{-2Nz_d}z(-2z_d) \sum_{\substack{R\subset Q\backslash\{d\}\\|R| {\rm \;even}}}  \left( \prod_{\ell\in R^c} H_{\{d\}}(z_\ell)\right) \left(\sum_{R=(A:B)} \prod_{j=1}^{|R|/2} H_{\{d\}}(z_{a_j},z_{b_j})\right)\nonumber \\
 &&+\sum_{d,g\in Q} e^{-2N(z_d+z_g)}\frac{z(z_d+z_g) z(-z_d-z_g)z(-2z_d)z(-2z_g)} {z(z_d-z_g) z(z_g-z_d)} \nonumber \\
 && \qquad\qquad \times \sum_{\substack{R\subset Q\backslash\{d,g\}\\|R| {\rm \;even}}}  \left( \prod_{\ell\in R^c} H_{\{d,g\}}(z_\ell)\right) \left(\sum_{R=(A:B)} \prod_{j=1}^{|R|/2} H_{\{d,g\}}(z_{a_j},z_{b_j})\right) .
\end{eqnarray}
Once again,  in the definition of 
\begin{equation}
H_{\{d,g\}}(z_\ell)= \frac{z'}{z}(z_\ell-z_d) -\frac{z'}{z}(z_\ell+z_d)+ \frac{z'}{z}(z_\ell-z_g) -\frac{z'}{z}(z_\ell+z_g) +\frac{z'}{z}(2z_\ell)
\end{equation}
 there is the term $\frac{z'}{z}(2z_\ell)=:H_{\emptyset}(z_\ell)$.  Also it is clear that $H_{\{d,g\}}(z_{a_j},z_{b_j})=H_{\emptyset}(z_{a_j},z_{b_j})$. As in the previous section, we rearrange terms, grouping together the  $H_{\emptyset}(z_\ell)$ terms.  Note that we could also pull together the $H_{\emptyset}(z_{a_j},z_{b_j})$ terms that occur in each line of (\ref{eq:J3}) to save a bit of space, but   here we will leave them separated to make the end result as similar to Gao's theorem as possible, to aid those who have worked with his result.  So we have that (\ref{eq:J3a}) equals
\begin{eqnarray}\label{eq:mess1}
&&\sum_{Q\cup M = \{1,\ldots,n\}} \Bigg( \prod_{m\in M} \int_{-\infty}^\infty f_m(x)dx\Bigg)\nonumber \\ &&\qquad \times \Bigg(\lim_{N\rightarrow \infty}  \sum_{S_2\subset Q} \left( \prod_{\ell \in S_2^c} \frac{2}{2\pi i} \int_{(\delta)}\frac{z'}{z}(2z_\ell)f_\ell\Big( \tfrac{N}{\pi}iz_\ell\Big)dz_\ell\right) \nonumber \\
&& \qquad \times\Bigg[\frac{1+(-1)^{|S_2|}}{2}\sum_{S_2 = (A:B)} \prod_{j=1}^{|S_2|/2}\frac{4}{(2\pi i)^2}\int_{(\delta)} \int_{(\delta)}  \left(\frac{z'}{z}\right)'(z_{a_j}+z_{b_j})f_{a_j}\Big( \tfrac{N}{\pi}iz_{a_j}\Big)f_{b_j}\Big( \tfrac{N}{\pi}iz_{b_j}\Big)dz_{a_j}dz_{b_j} \nonumber \\
&&\qquad+\sum_{\substack{S_3\subsetneq S_2\\|S_3|{\rm \;even}}} \left(\sum_{S_3=(C:D)} \prod_{j=1}^{|S_3|/2}\frac{4}{(2\pi i)^2} \int_{(\delta)} \int_{(\delta)} \left(\frac{z'}{z}\right)'(z_{c_j}+z_{d_j})f_{c_j}\Big( \tfrac{N}{\pi}iz_{c_j}\Big)f_{d_j}\Big( \tfrac{N}{\pi}iz_{d_j}\Big)dz_{c_j}dz_{d_j}\right)\nonumber \\
&&\qquad \times \Bigg\{\sum_{\substack{I\cup I_c = S_3^c\\|I_c|\geq 1}} \sum_{d\in I^c}\frac{2}{2\pi i}\int_{(\delta)} \left(-e^{-2Nz_d}z(-2z_d)\right) \left( \prod_{i\in I} \frac{2}{2\pi i}\int_{(\delta)}-\frac{z'}{z}(z_i+z_d)f_i\Big( \tfrac{N}{\pi}iz_i\Big) dz_i\right) \nonumber \\
&&\qquad\times\left(\prod_{j\in I^c\backslash\{d\}}\frac{2}{2\pi i} \int_{(\delta)} \frac{z'}{z}(z_j-z_d)f_j\Big( \tfrac{N}{\pi}iz_j\Big) dz_j \right) f_d\Big( \tfrac{N}{\pi}iz_d\Big) dz_d \Bigg\}  \nonumber\\
&& \qquad+\sum_{\substack{S_4\subsetneq S_2\\|S_4|{\rm \;even}}} \left(\sum_{S_4=(G:H)} \prod_{j=1}^{|S_4|/2}\frac{4}{(2\pi i)^2} \int_{(\delta)} \int_{(\delta)} \left(\frac{z'}{z}\right)'(z_{g_j}+z_{h_j})f_{g_j}\Big( \tfrac{N}{\pi}iz_{g_j}\Big)f_{h_j}\Big( \tfrac{N}{\pi}iz_{h_j}\Big)dz_{g_j}dz_{h_j}\right)\nonumber \\
&&\qquad \times \Bigg\{\sum_{\substack{I_1\cup I_2\cup I_1^c\cup I_2^c =S_4^c\\ |I_1^c|\geq 1,\; |I_2^c|\geq 1}} \sum_{d\in I_1^c}\sum_{g \in I_2^c}\frac{4}{(2\pi i)^2}\int_{(\delta)}\int_{(\delta)} \left(e^{-2N(z_d+z_g)}\frac{z(z_d+z_g) z(-z_d-z_g)z(-2z_d)z(-2z_g)} {z(z_d-z_g) z(z_g-z_d)}\right) \nonumber\\
&&\qquad\qquad \qquad\qquad\times \left( \prod_{i\in I_1} \frac{2}{2\pi i}\int_{(\delta)}-\frac{z'}{z}(z_i+z_d)f_i\Big( \tfrac{N}{\pi}iz_i\Big) dz_i\right) \nonumber \\
&&\qquad\qquad\qquad\qquad\times\left(\prod_{j\in I_1^c\backslash\{d\}}\frac{2}{2\pi i} \int_{(\delta)} \frac{z'}{z}(z_j-z_d)f_j\Big( \tfrac{N}{\pi}iz_j\Big) dz_j \right)\nonumber \\
&&\qquad\qquad\qquad\qquad \times \left( \prod_{i\in I_2} \frac{2}{2\pi i}\int_{(\delta)}-\frac{z'}{z}(z_i+z_g)f_i\Big( \tfrac{N}{\pi}iz_i\Big) dz_i\right) \nonumber \\
&&\qquad\qquad\qquad\qquad\times\left(\prod_{j\in I_2^c\backslash\{g\}}\frac{2}{2\pi i} \int_{(\delta)} \frac{z'}{z}(z_j-z_g)f_j\Big( \tfrac{N}{\pi}iz_j\Big) dz_j \right)\nonumber \\
&&\qquad\qquad\qquad\qquad\times f_d\Big( \tfrac{N}{\pi}iz_d\Big) f_g\Big(\tfrac{N}{\pi} iz_g\Big)dz_d dz_g\left. \right)  \nonumber
\Bigg\}\Bigg]  \Bigg).
\end{eqnarray}

We now use Lemma \ref{lem:5} for the last six lines of (\ref{eq:mess1}), and noting that the first seven lines of (\ref{eq:mess1}) are all familiar from the previous section, our final result is

\begin{theorem}\label{theo:support3}
For test functions $f_1, \ldots, f_n$ with properties (\ref{eq:testfunction}) and with the product of their Fourier transforms $\prod_{i=1}^n \hat{f_i}(u_i)$ having support in $\sum_{i=1}^n |u_i|<3$, 
\begin{eqnarray}
&&  \lim_{N\rightarrow \infty}  \int_{USp(2N)} \sum _{\substack{j_1,\ldots,j_n=-\infty \\j_1,\ldots,j_n\neq 0}}^\infty f_1\Big(\frac{N}{\pi} \theta_{j_1}\Big) \cdots f_n\Big(\frac{N}{\pi} \theta_{j_n}\Big) d\mu_{{\rm Haar}}\nonumber \\
&&=\sum_{Q\cup M = \{1,\ldots,n\}} \Bigg( \prod_{m\in M} \int_{-\infty}^\infty f_m(x)dx\Bigg) \times \Bigg(  \sum_{S_2\subset Q} \left( \Big(\frac{-1}{2}\Big)^{|S_2^c|} \prod_{\ell \in S_2^c} \int_{-\infty}^{\infty} \hat{f_\ell}(u)du\right) \nonumber \\
&& \qquad \times\Bigg[\frac{1+(-1)^{|S_2|}}{2}2^{|S_2|/2} \sum_{S_2=(A:B)} \prod_{i=1}^{|S_2|/2} \int_{-\infty}^\infty |u_i| \hat{f_{a_i}} (u_i) \hat{f_{b_i}}(u_i) du_i  \nonumber \\
&&\qquad-\frac{1}{2}\sum_{\substack{S_3\subsetneq S_2\\|S_3|{\rm \;even}}} 2^{|S_3|/2} \left( \sum_{S_3=(C:D)} \prod_{i=1}^{|S_3|/2} \int_{-\infty}^\infty |u_i| \hat{f_{c_i}}(u_i)\hat{f_{d_i}}(u_i) du_i\right)\nonumber \\
&&\qquad\qquad \times \Bigg\{\sum_{\substack{I\cup I_c = S_3^c\\|I_c|\geq 1}} (-1)^{|I^c|} 2^{|S_3^c|}\int_{\substack{(\mathbb{R}\geq 0)^{S_3^c} \\ \sum_{i\in I}u_i\leq (\sum_{i\in I^c} u_i)-1}} \prod_{i\in S_3^c} \hat{f_i}(u_i) \prod_{i\in S_3^c} du_i \Bigg\}  \nonumber\\
&& \qquad+\sum_{\substack{S_4\subsetneq S_2\\|S_4|{\rm \;even}}}2^{|S_4|/2} \left( \sum_{S_4=(G:H)} \prod_{i=1}^{|S_4|/2} \int_{-\infty}^\infty |u_i| \hat{f_{g_i}}(u_i)\hat{f_{h_i}}(u_i) du_i\right)\nonumber \\
&&\qquad\qquad \times \Bigg\{\sum_{\substack{I_1\cup I_2\cup I_1^c\cup I_2^c =S_4^c\\ |I_1^c|\geq 1,\; |I_2^c|\geq 1}}(-1)^{|I^c_1\cup I^c_2|} 2^{|S_4^c|}\int_{\substack{(\mathbb{R}\geq 0)^{|S_4^c|} \\ \sum_{i\in I_1}u_i\leq (\sum_{i\in I^c_1} u_i)-1\\ \sum_{i\in I_2}u_i\leq (\sum_{i\in I^c_2} u_i)-1}} \left( \frac{1}{4}-\left[\left(-\sum_{i\in I_2}u_i+\sum_{j\in I^c_2} u_j-1\right)\right.\right.\nonumber \\
&&\qquad\qquad \times\left.\left. \delta\left(-\sum_{i\in I_1\cup I^c_2}u_i+\sum_{j\in I^c_1\cup I_2} u_j\right)\right]\right)\prod_{i\in S_4^c} \hat{f_i}(u_i) \prod_{i\in S_4^c} du_i \nonumber
\Bigg\}\Bigg]  \Bigg)
\end{eqnarray}
where the eigenangles $\theta_{j_i}$ are counted as described at Theorem \ref{theo:massna}. Here $Q\cup M=\{1,\ldots,n\}$ denotes partitioning $\{1,\ldots,n\}$ into two disjoint sets, and analogously for other instances of similar notation.  $S_2^c$ denotes the complement of $S_2$ in $Q$. The sum denoted $(A:B)$ indicates summing over all partitions $\{\{a_1,b_1\},\ldots\{a_{|S_2|/2},b_{|S_2|/2}\}\}$ of  $S_2$.
\end{theorem}
In the above theorem we use the convention that empty products are 1 and empty sums are 0, but note that the empty set is an allowed set in a sum over sets.  In particular, if $S_3$ is empty then the sum over $(C:D)$ is 1.  If $Q$ is empty then the sum over $S_2$ is 1. If $S_2$ is empty then the sum over $S_3$ is zero, and so forth.

\pagebreak


\end{document}